\documentclass[12pt]{amsart}

\usepackage{amsmath, amssymb, graphics, setspace}
\usepackage[usenames]{color}
\input xypic

\usepackage{algorithm}
\usepackage[noend]{algpseudocode}

\algblock{Input}{EndInput}
\algnotext{EndInput}
\algblock{Output}{EndOutput}
\algnotext{EndOutput}
\newcommand{\Desc}[2]{\State \makebox[2em][l]{#1} #2}

\renewcommand*\Call[2]{\textproc{#1}(#2)}

\makeatletter
\def\BState{\State\hskip-\ALG@thistlm}
\makeatother

\usepackage{verbatim}
\usepackage{url}
\usepackage{bm}
\usepackage{tikz}
\usepackage{algorithm}
\usepackage[noend]{algpseudocode}
\renewcommand*\Call[2]{\textproc{#1}(#2)}
\newcommand{\mathsym}[1]{{}}
\newcommand{\unicode}[1]{{}}

\makeatletter
\def\BState{\State\hskip-\ALG@thistlm}
\makeatother

\usepackage{amsmath}
\usepackage[centertags]{amsmath}
\usepackage{amsfonts}
\usepackage{amssymb}
\usepackage{amsthm}
\usepackage{newlfont}
\usepackage{url}
\usepackage{bm}
\usepackage{chngcntr}
\theoremstyle{plain}
\newtheorem{thm}{Theorem}
\newtheorem{cor}[thm]{Corollary}

\newtheorem{lem}[thm]{Lemma}
\newtheorem{lem*}[thm]{Lemma}
\newtheorem{prop}[thm]{Proposition}

\theoremstyle{definition}

\theoremstyle{remark}
\newtheorem{rem}{Remark}
\newtheorem{rem*}{Remark}

\newtheorem{example}[rem]{Example}

\numberwithin{rem}{section} 
\numberwithin{dfn}{section} 
\numberwithin{equation}{section} 
\numberwithin{thm}{section} 

\def\!{\operatorname{!}}

\def\F{\mathbb F}
\def\G{\mathbb G}

\def\1{\bold 1}

\def\deg{\operatorname{deg}}

\def\Hom{\operatorname{Hom}}

\def\Ext{\operatorname{Ext}}

\def\rk{\operatorname{rk}}

\def\deg{\operatorname{deg}}

\def\End{\operatorname{End}}

\usepackage{amssymb}

\usepackage{enumitem}

\usepackage{graphicx}
\usepackage{xy}
\usepackage{float}

\makeatletter
\@namedef{subjclassname@2020}{%
	\textup{2020} Mathematics Subject Classification}
\makeatother

\usepackage[T1]{fontenc}

\newtheorem{theorem}{Theorem}[section]

\theoremstyle{definition}
\newtheorem{definition}[theorem]{Definition}
\newtheorem{remark}[theorem]{Remark}

\usepackage[foot]{amsaddr}

\numberwithin{equation}{section}

\newcommand{\Der}{\mathrm{Der} }

\newcommand{\Derin}{\Der_{in}}
\newcommand{\up}[2]{#1^{(#2)}}

\newcommand{\lra}{\longrightarrow}

\newcommand{\podwzorem}[2]{\underbrace{#1}\limits_{#2}}
\newcommand{\nadwzorem}[2]{\overbrace{#1}\limits^{#2}}

\newcommand{\uplra}[1]{\stackrel{#1}{\lra}}

\begin{document}

	\baselineskip=17pt
	
	
	\title[Algorithms for determination \dots]{Algorithms for determination of $\mathbf{T}$-module structures on some  extension  groups   }

\author[F. G{\l}och]{Filip G{\l}och$^{1}$}
\address{$^1$ Institute of Mathematics, Department of Exact and Natural Sciences, University of Szczecin, ul. Wielkopolska 15, 70-451 Szczecin, Poland }\email{filip.gloch@gmail.com}
\address{$^{\ast}$Corresponding author}
\author[D.E. K{\k e}dzierski]{Dawid E. K{\k e}dzierski $^{1}$}\email{dawid.kedzierski@usz.edu.pl}
\author[P. Kraso{\'n}]{Piotr Kraso{\'n} $^{1,{\ast}}$}\email{piotrkras26@gmail.com}

	\date{\today}
	

	
\thanks{}

\maketitle

\newcommand{\tm}{$\mathbf{t}-$}
\newcommand{\tsm}{$\mathbf{t}^{\sigma}-$}
\newcommand{\functor}{\mathrm{Rep}}

\renewcommand{\rk}{\mathrm{deg}_\tau}


\begin{abstract}
In  their recent work the second and the third authors extended the methods of  M.A. Papanikolas  and N. Ramachandran and determined the {\bf t}-module structure 
on $\Ext^1(\Phi,\Psi )$ where $\Phi $ and $\Psi$ were Anderson {\bf t}-modules   over $A={\mathbf F}_q[t]$ of some specific types.
This approach involved the concept of biderivation and certain reduction algorithm.
In this paper  we generalize  these results   and 
present a complete algorithm for computation of {\bf t}-module structure on $\Ext^1(\Phi,\Psi )$ 
for \tm modules $\Phi $ and $\Psi$ 
such that  $\deg_{\tau} \Phi >  \deg_{\tau} \Psi.$ The last condition is not sufficient for our algorithm to be executable. We show that it can be applied when the matrix at the biggest power of $\tau$ in $\Phi_t$ is invertible.
We also introduce a notion of  ${\tau}$-composition series which we find suitable for the additive category of {\bf t}-modules and show 
that  under certain assumptions on the composition series of $\Phi $ and $\Psi$ our algorithm is also executable. In this version we added a Corrigendum the end of the paper.

\end{abstract}

\section{Introduction}

Let $A={\mathbb F}_q[t]$ be a polynomial ring over a finite field ${\mathbb F}_q$ with $q$-elements 
and let $k={\mathbb F}_q(t)$ be the quotient field of $A.$
Then   $k$ can be viewed as the function field of the curve ${\mathbb P}^1/{\mathbb F}_q$ and $A$ as the ring of functions regular outside $\infty$.
Let  $v_{\infty}: k\rightarrow {\mathbb R}\cup\{\infty\}$ be the normalized
valuation associated to $\frac{1}{t}$ (i.e. $v_{\infty}(\frac{1}{t})=1$). Let $K$  be a completion of $k$ with respect to $v_{\infty}$ and let $\overline K$ be its fixed  algebraic closure. It turns out that ${\overline K}$ is neither complete nor 
locally compact. Denote also by $\overline{v_{\infty}}$ the extension of $v_{\infty}$ to $\overline K.$
Let  ${\mathbb C}_{\infty}$ be the completion of $\overline K$ with respect to the metric induced by  $\overline{v_{\infty}}.$

One of the major developments in the theory of curves over finite fields was the work of L. Carlitz \cite{c35} where
a  remarkable {\it exponential} function  $e_C : {\mathbb C}_{\infty}\rightarrow  {\mathbb C}_{\infty}$ was defined.
This function is an analytic function on the characteristic $p$-space ${\mathbb C}_{\infty},$ and besides the properties analogous to the usual exponential function satisfies the functional equation $e_C(az)=C_a(e_C(z))$ where 
$a\in{\mathbb F}_q[t]$ and $C_a$ is an additive polynomial. One has $C_{ab}(z)=C_{a}(C_b(z))=C_{b}(C_a(z))$
and $C_{a}(z_1+z_2)=C_{a}(z_1)+C_{a}(z_2).$  The correspondence $a\rightarrow C_a(z)$ is called the Carlitz module. 
D. Hayes \cite{h74} proved  that the division values of $e_C(z)$ generate abelian extensions of $k$ analogous to the cyclotomic extensions defined by the division points of $e^{2{\pi}iz}.$
In \cite{d74} V.G. Drinfeld extended the theory of the Carlitz  (rank 1) exponentials to arbitrary rank $d.$
Such exponentials 
give rise to objects which are now called Drinfeld modules. The theory of  moduli spaces of shtukas (which are generalizations of Drinfeld modules) constructed by V.G. Drinfeld played a major role in proving the version of Langlands correspondence for $GL_r$ of function fields in finite characteristic  \cite{l04},\cite{la02}. 
Additionally, D. Hayes \cite{h80} developed  explicitly a class field theory for function fields.
Drinfeld modules and its generalizations \tm modules defined by G. Anderson \cite{a} are objects of intensive studies.
For the remarkable results obtained in this area the reader is advised to consult excellent sources e.g. \cite{th04}, \cite{g96}, \cite{f13}, \cite{bp20}, \cite{ge86}, \cite{MP}. 
The category of \tm modules is an additive category. Basic structures of this category such as $\Hom_{\tau}$-sets 
or extension groups deserve thorough investigation. In \cite{kp22} an algorithm for determination of an endomorphism 
ring of a Drinfeld module was developed.
 In  \cite{pr} a \tm module structure  for $\Ext^1_{\tau}(\phi,C)$, where $\phi$ is a Drinfeld module of rank bigger than $1$  and $C$ is the   Carlitz module, was determined.
The key idea relied on representation of the ${\mathbb F}_q$-space $\Ext^1_{\tau}$ as the quotient of the space 
of biderivations modulo the subspace of inner biderivations. Then the condition  ${\mathrm{rk}}\phi >{\mathrm{rk}} C$ on a Drinfeld module allowed the authors of 
\cite{pr} to construct a recursive algorithm for computing \tm module structure on $\Ext^1_{\tau}(\phi, C).$
The application of essentially the same algorithm allowed them to determine the structure of a \tm module on the space 
$\Ext^1_{\tau}(C^{\otimes m}, C^{\otimes n})$ for $n>m$. 
In  \cite{kk04} the second and third author  generalized this algorithm and showed that for some specific \tm modules $\Phi$ and $\Psi$ under the assumption 
$\rk\Phi>\rk\Psi$   the spaces $\Ext^1_{\tau}(\Phi,\Psi)$ have   \tm module structures. The content of this paper is the following:  In section \ref{sec2} we give basic definitions and facts concerning \tm modules, their morphisms etc. Then we describe $\Ext^1_{\tau}$ as the quotient  of the space of biderivations modulo the subspace of inner biderivations. We  quickly recall the six-term $\Hom_{\tau} - \Ext_{\tau}$ exact sequence existence of which was  proved in \cite{kk04}. We  also describe the duality functor as defined in \cite{g95} (cf. \cite{kk04}). 

 In section \ref{sec 3} we present an algorithm, which we call  \tm reduction algorithm, for computation of a \tm module structure for  $\Ext^1_{\tau}(\Phi,\Psi),$ where
$\Phi, \Psi$ are \tm modules, $\rk\Phi>\rk\Psi$.  The \tm reduction algorithm is a generalization of the methods of 
computation presented both in \cite{pr} and \cite{kk04}.
The  condition $\rk\Phi>\rk\Psi$ is not sufficient for the \tm reduction algorithm to be executable therefore we introduce the notion of a \tm module structure coming from \tm reduction. 
In section \ref{sec4} we show that if $\rk\Phi>\rk\Psi$ and the matrix $A_n$ at the highest power of $\tau$ in ${\Phi}_t$ is invertible then 
$\Ext^1_{\tau}(\Phi,\Psi)$ has a \tm module structure coming from \tm reduction. We  then present an appropriate pseudo-code. In section \ref{sec5} we introduce for \tm modules some new notions: we define $\tau$-simplicity and $\tau$-composition series. Then we show that our algorithm is executable for $\Ext_{\tau}^1(\Phi, \Psi)$ where  $\Phi$ and $\Psi$ are given 
 by  composition series  having  Drinfeld modules as consecutive sub-quotients which satisfy   certain conditions on degrees i.e. in this case the \tm module structure 
 on $\Ext_{\tau}^1(\Phi, \Psi)$ comes from \tm reduction. Then we present a  pseudo-code for this situation.
 This essentially enlarges the class of \tm modules for which the structure of a \tm module on $\Ext_{\tau}^1(\Phi, \Psi)$ 
 can be effectively computed.
 For theoretical purposes it might be also useful to have exact formulas. This in general  is a very tedious task.
 We give the exact formulas for the case of $\Ext^1_{\tau}(\phi,\psi)$ for two Drinfeld modules with $\rk\phi>\rk\psi$. 
 We also derive some consequences from them which cannot be obtained directly from the 
   algorithm. Finally, in the appendix we give an implementation of our pseudo-codes in Mathematica 13.2 and compute 
   two examples.
   
   \section{Preliminaries}\label{sec2}
 In this section we give basic definitions and properties concerning Drinfeld modules and \tm modules.

 Let $p$ be a rational prime and $A={\mathbb F}_q[t]$  the polynomial ring over the finite field with $q=p^m$ elements.
\begin{definition}\label{def1}
An $A$-field $K$ is a fixed morphism  
 ${\iota}: A \rightarrow K$. The kernel of $\iota$ is a prime ideal $\cal P$ of $A$ called the characteristic. The characteristic of $\iota$ is called finite if ${\cal P}\neq 0,$ or generic (zero) if ${\cal P}= 0.$
\end{definition}
Let $\G_{a,K}$ be the additive algebraic group over $K.$
Then the endomorphism ring ${\End}(\G_{a,K})$ is the skew polynomial ring $K\{\tau\}.$ The endomorphism $\tau$ is 
the map $u\rightarrow u^q$ and therefore one has the commutation relation $\tau u=u^q\tau$ for $u\in K.$
\begin{definition}\label{def2}
A Drinfeld $A$-module is a  homomorphism $\phi : A\rightarrow K\{\tau\},$ $ a\rightarrow {\phi}_{a}\,,$ of ${\mathbb F}_{q}$-algebras such that 
\begin{enumerate}
\item[1.] $D\circ {\phi}={\iota},$
\item[2.] for some $a\in A, \, {\phi}_{a}\neq {\iota}(a){\tau}^{0},$
\end{enumerate}
where $D({\sum}_{i=0}^{i={\nu}}\,\,a_{i}{\tau}^{i})=a_{0}. $ The characteristic of a Drinfeld module is the characteristic of ${\iota}.$
\end{definition}
There exists a generalization of the notion of a Drinfeld module called a \tm module. This was developed by G. Anderson in \cite{a}.
\begin{definition}
An $m$-dimensional \tm module over $K$ is an ${\mathbb F}_q$ -algebra homomorphism
\begin{equation}\label{tmodule}
{\Phi} : {\mathbb F}_q[t]\rightarrow {\mathrm{Mat}}_m (K\{\tau\})
\end{equation}
such that ${\Phi}(t)$, as a polynomial in $\tau$  with coefficients in ${\mathrm{Mat}}_m(K)$  is of the following form:
\begin{equation}\label{tmod1}
{\Phi}(t)=({\theta}I_d+N)\tau^0+M_1{\tau}^1+\dots +M_r\tau^r
\end{equation}
where $I_m $ is the identity matrix and $N$ is a nilpotent matrix. In general, a \tm module over $K$ is an algebraic group $E$ defined over $K$ and isomorphic over $K$ to $\G^m_a$  together with a choice of ${\mathbb{F}}_q$-linear endomorphism $t:E\rightarrow E$ such that $d(t-\theta)^n{\mathrm{Lie}}(E)=0$ for $n$ sufficiently large. Notice that 
$d(\cdot)$ stands here for the differential of an endomorphism of algebraic groups. The choice of  an isomorphism $E\cong \G^m_a$ is equivalent to the choice of $\Phi.$
\end{definition} 
Notice that a Drinfeld module may be viewed as a one dimensional \tm module.
\begin{definition}
Let $\Phi$ be a \tm module defined over $K$ as in (\ref{tmodule})  and $L$ be an algebraic extension of $K.$ The Mordell-Weil group ${\Phi}(L)$ is the additive group of $L^m$ viewed as an $A$-module via evaluation of the matrix   ${\Phi}_{a}\in {\mathrm{Mat}}_m(K\{\tau\}),\,\, a\in A$ on $L^m,$ where $\tau$ acts on L as the  Frobenius morphism $x\rightarrow x^q.$
\end{definition}
To make notation simpler for a non-negative integer $n$ we denote $x^{q^n}:=x^{(n)}$ to be  the evaluation of the Frobenius twist ${\tau}^n$ on $x\in K.$ Of course $x^{(0)}=x$.
\begin{definition}\label{rank}
Let $\Phi$ be a \tm module. Then
\begin{itemize}
\item[(i)]
  the rank of $\Phi$ is  defined as the rank of the period lattice of $\Phi$ as a $d{\Phi}(A)$-module (cf. \cite[Section t-modules]{bp20}). 
 \item[(ii)] the degree $\deg_{\tau}\Phi$ of a \tm module $\Phi$ is defined as $\deg_{\tau}\Phi_t.$ 
 \end{itemize}
 \end{definition}
 \begin{remark}
      Notice that for Drinfeld modules $\mathrm{rk}\phi=\deg_{\tau}\phi$ but for a general \tm module the rank of $\Phi$  is not always equal to $\deg_{\tau}\Phi$  (cf. \cite[Example 5.9.9]{g96}).
  \end{remark}

\begin{definition} 
	Let $\Phi$ and $\Psi$ be  two \tm modules of dimension $d$ and $e$, respectively. A morphism $f:\Psi\lra \Phi$ of \tm modules over $K$ is a matrix $f\in\mathrm{Mat}_{d\times e}(K\{\tau\})$
	 such that
	$$
	f\Psi(t) = \Phi(t)f.
	$$
\end{definition}
\begin{remark}\label{sub}
Notice that the category of \tm modules is an additive subcategory of the abelian category of ${\mathbb F}_q[t]-$modules. This 
subcategory is not full (cf. \cite[Example 10.2]{kk04}). Therefore the $\Hom$-set in the category of \tm modules will be denoted 
as ${\Hom}_{\tau}$ i.e. $\Hom(\Phi,\Psi):={\Hom}_{\tau}(\Phi, \Psi).$ 
\end{remark}
By a zero \tm module we mean the \tm module given by the map ${\mathbb F}_q[t]\rightarrow 0.$
Let $\mathrm{Ext}^1_{\tau}(\Phi, \Psi)$ be the Baer group of extensions of \tm modules i.e. the group of exact sequences
\begin{equation}\label{seq}
0\rightarrow \Psi\rightarrow \Gamma\rightarrow \Phi\rightarrow 0
\end{equation}
with the usual addition of extensions (cf. \cite{Mac}).

An extension of a $\mathbf t$-module $\Phi:\F_q[t]\lra {\mathrm{Mat}}_d(K\{\tau\})$ by $\Psi:\F_q[t]\lra {\mathrm{Mat}}_e(K\{\tau\})$ can be determined by a biderivation i.e. $\F_q-$linear map $\delta:\F_q[t]\lra {\mathrm{Mat}}_{e\times d}(K\{\tau\})$ such that
\begin{equation}\label{delta}	
	\delta(ab)=\Psi(a)\delta(b)+\delta(a)\Phi{(b)}\quad \textnormal{for all}\quad a,b\in\F_q[t].
\end{equation}
The $\F_q-$vector space of all biderivations will be denoted by $\Der(\Phi, \Psi)$ (cf. \cite{pr}). 
 The map $\delta\mapsto \delta(t)$ induces the $\F_q-$linear isomorphism of the vector spaces $\Der(\Phi, \Psi)$ and $ {\mathrm{Mat}}_{e\times d}(K\{\tau\})$.
  Let $\delta^{(-)}: {\mathrm{Mat}}_{e\times d}(K\{\tau\})\lra \Der(\Phi, \Psi)$ be an $\F_q-$linear map defined by the following formula:
	$$\delta^{(U)}(a)=U\Phi_a - \Psi_aU\quad \textnormal{for all}\quad a\in \F_q[t]\quad\textnormal{and}\quad U\in {\mathrm{Mat}}_{e\times d}(K\{\tau\}).$$
	The image of the map  $\delta^{(-)}$ is denoted by $\Derin(\Phi, \Psi)$, and is  called the space of inner biderivations. 
We have the following $\F_q[t]-$module isomorphism (cf.  \cite[Lemma 2.1]{pr}):
	\begin{align}\label{iso_ext}
		\Ext^1_{\tau}(\Phi,\Psi)\cong\mathrm{coker}\delta^{(-)}=\Der(\Phi, \Psi)/\Derin(\Phi, \Psi).
	\end{align}

In \cite{kk04} the second and third author proved the following theorem:

\begin{thm}\cite[Theorem 10.2]{kk04}\label{thm:long_sequence}
Let 
$$\delta: \quad 0 \lra F\uplra{i} X\uplra{\pi} E\lra 0$$
 be a short exact sequence of \tm modules given by the biderivation $\delta$ and let $G$ be a \tm module. 
\begin{itemize}
	\item[$(i)$] There is an exact sequence of $\F_q[t]-$modules:
	\begin{align*}
		0\lra \Hom_{\tau}(G,F)\uplra{i\circ-}  \Hom_{\tau}(G,X)\uplra{\pi\circ-}  \Hom_{\tau}(G,E)\lra  \\
		\uplra{\delta\circ-} \Ext^1_{\tau}(G,F)\uplra{-i\circ-}  \Ext^1_{\tau}(G,X)\uplra{-\pi\circ-}  \Ext^1_{\tau}(G,E)\rightarrow 0.
	\end{align*}
	\item[$(ii)$]  There is an exact sequence of $\F_q[t]-$modules:
	\begin{align*}
		0\lra \Hom_{\tau}(E,G)\uplra{-\circ \pi}  \Hom_{\tau}(X,G)\uplra{-\circ i}  \Hom_{\tau}(F,G)\uplra{-\circ\delta}  \\
		\lra \Ext^1_{\tau}(E,G)\uplra {-\circ (-\pi)} \Ext^1_{\tau}(X,G)\uplra{-\circ (-i)}  \Ext^1_{\tau}(F,G)\rightarrow 0. 
	\end{align*}
\end{itemize} 
\end{thm}

When $K$ is a perfect field,  $\sigma$ will denote the inverse of $\tau$.  Then to simplify notation the value of $\sigma^n$ on $x\in K$ will be denoted as $x^{(-n)}$ for $n\in {\mathbb N}.$ 

\begin{definition}\label{twist}
Let $K$ be a perfect field. The ring $K\{\sigma\}$  of adjoint twisted polynomials over $K$ is defined by the following action of $\sigma$ (cf. \cite{g95}):
\begin{equation*}\label{ad1}
	{\sigma}x=x^{(-1)}\sigma \quad \textnormal{for}\quad x\in K.
\end{equation*}
\end{definition}
 A  \tsm module is defined similarly as a \tm module  by replacing $\tau$ with $\sigma$. Similarly one defines  a morphism of  \tsm modules.   The category of \tsm modules  with the zero  \tsm module attached is an additive, $\F_q[t]-$linear category.

 The following maps:
\begin{equation*}\label{maps1}
(-)^{\sigma}: K\{\tau\}\rightarrow K\{\sigma\};  \quad \Big(\sum_{i=0}^na_i\tau^i\Big)^{\sigma}=\sum_{i=0}^na_i^{(-i)}{\sigma}^i,
\end{equation*}
\begin{equation*}\label{maps2}
(-)^{\tau}: K\{\sigma\}\rightarrow K\{\tau\};  \quad \Big(\sum_{i=0}^nb_i\sigma^i\Big)^{\tau}=\sum_{i=0}^nb_i^{(i)}{\tau}^i.
\end{equation*}
are ${\mathbb F}_q$-linear mutual inverses.
 We associate with $\Phi : {\mathbb F}_q[t]\rightarrow {\mathrm{Mat}}_e(K\{\tau\})$ the adjoint homomorphism 
\begin{equation*}\label{ad}
\Phi^{\sigma}:{\mathbb F}_q[t]\rightarrow {\mathrm{Mat}}_e(K\{\sigma\})
\end{equation*}
such that    each matrix  $X_t$ is mapped  to $\Big[\big(X_t(\tau)\big)^{\sigma}\Big]^T.$  The inverse of $(-)^{\sigma}$ is given 
by the map that associates with  $\Gamma : {\mathbb F}_q[t]\rightarrow {\mathrm{Mat}}_e(K\{\sigma\})$ 
the following homomorphism:
\begin{equation*}\label{ad1}
\Gamma^{\tau}:{\mathbb F}_q[t]\rightarrow {\mathrm{Mat}}_e(K\{\tau\}); \quad X_t \rightarrow \Big[\big(X_t(\sigma)\big)^{\tau}\Big]^T
\end{equation*}
We have the following:
\begin{thm}\cite[Theorem 7.2]{kk04}\label{Duality}
	Assume that $K$ is a perfect $A-$field. 
Let $\Phi$ and $\Psi$ be \tm modules. Then there exists an isomorphism of ${\mathbb F}_q[t]$-modules:
\begin{equation*}\label{diso}
\Ext^1_{\tau}(\Phi,\Psi)\cong \Ext^1_{\sigma}(\Psi^{\sigma},\Phi^{\sigma})
\end{equation*} 
\end{thm} 
  Let us make the following general remark:
\begin{remark}\label{remk}
By Theorem \ref{Duality} all algorithms presented in the sequel can be easily adapted for computation
of \tsm structure on suitable $\Ext^1_{\tau}(\Phi ,\Psi )$ where $\rk\Psi>\rk\Phi.$
\end{remark}

\section{Algorithm of  \tm reduction}\label{sec 3}

 In \cite{pr} a method for determination of  \tm module structure 
on  \\ $\Ext^1(C^{\otimes n}, C^{\otimes m})$ for $n<m$ was given. The method used the description of extension groups in terms of biderivations. Then  various generalizations of this were constructed and used in \cite{kk04} for determination 
of the \tm module structure on the  space $\Ext^1_{\tau}(\Phi, \Psi)$ where $\Phi$ and $\Psi$ are  \tm modules of some specific forms satisfying the condition $\rk\Phi>\rk \Psi$.
It turns out that these computations can be put together into one algorithm which we call \tm   reduction.
We advise the reader to consult the  Example 4.1 of \cite{kk04}  which illustrates the action of the \tm reduction algorithm 
in the special case of two Drinfeld modules.

Before we describe the \tm reduction algorithm in full generality let us introduce some more notation.

Let  $\Phi$ resp. $\Psi$ be \tm modules of dimensions $d$ resp. $e$. By  $E_{i\times j}$ we denote the matrix   of type $e\times d$, where the only nonzero entry is  $1$ at the place $i\times j$. 
For a matrix $N=\big[n_{i,j}\big]\in\mathrm{Mat}_{e\times d}(\mathbb{Z}_{\geq 0})$, whose entries are non-negative integers we denote:
$$\mathrm{Mat}_{e\times d}(K\{\tau\})_{< N}=\Big\{	
\big[w_{i,j}(\tau)\big]\in \mathrm{Mat}_{e\times d}(K\{\tau\})\mid \deg_{\tau}w_{i, j}(\tau)<n_{i, j}\quad \forall i,j
\Big\}$$

The \tm reduction algorithm is performed in the following steps \medskip \\
\textbf{Step 1.} \smallskip\\
We use an isomorphism of  $\F_q[t]-$modules  
$$\Ext^1_{\tau}(\Phi, \Psi)\cong \Der(\Phi, \Psi)/\Derin(\Phi, \Psi),$$
and identify every element of  $\Ext^1_{\tau}(\Phi,\Psi)$ with the corresponding (class) of biderivation $\delta:\F_q[t]\lra \mathrm{Mat}_{e\times d}(K\{\tau\})$. Since every biderivation is uniquely determined by  $\delta_t:=\delta(t)$,
we can view every element of 
  $\Ext^1_{\tau}(\Phi,\Psi)$ as the class of the matrix $\delta_t\in\mathrm{Mat}_{e\times d}(K\{\tau\})$. \medskip \\
\textbf{Step 2.} \smallskip\\
 We find a basis $\big(U_{i\times j}\big)$ of the space $\mathrm{Mat}_{e\times d}(K)$ such that the inner biderivations 
  $\delta^{(c\tau^kU_{i\times j})}$ for $c\in K$ i $k=0,1,2,\dots$  fulfill the following conditions:
 \begin{itemize}
 	\item[A.] the inner biderivations $\delta^{(c\tau^k U_{i\times j})}$ form a basis of the  
  $\F_q-$linear space $\Derin(\Phi, \Psi)$;
 	\item[B.] the matrix $\delta^{(c\tau^kU_{i\times j})}_t$ has at the $i\times j$-entry  a polynomial whose coefficient 
	at the biggest power of  $\tau$ is equal to $c$. 
 	\item[C.] for every biderivation $\delta\in \Der(\Phi,\Psi)$ there exists a sequence of inner biderivations  $\delta^{(c_l\tau^{k_l} U_{i\times j})}$ wbere $c_l\in K$ and $k_l\in\mathbb{Z}_{\geq 0}$   for $l=1,2,\dots, r$ 
	such that the reduced biderivation:
 		$$\delta^{\mathrm{reduce}}=\delta -\sum_{l=1}^r \delta^{(c_l\tau^{k_l}\cdot U_{i\times j})}\in \mathrm{Mat}_{e\times d}(K\{\tau\})_{< N}, $$
 		where
 		$N=\Big[ 
 		\deg_{\tau}\delta^{(c\tau^0U_{i\times j})}_t\Big]_{i,j}.$
 		
 	 \item[D.]  different reduced biderivations in 
 	  $\mathrm{Mat}_{e\times d}(K\{\tau\})_{< N}$ represent different elements in $\Ext^1_{\tau}(\Phi,\Psi)$. 
 			 
 \end{itemize}
Points  A., B., C. and  D. imply that there exists an isomorphism of  $\F_q-$linear spaces:
\begin{equation}
	\label{eq:algorytm_reduckji_step2}
	\Ext^1_{\tau}(\Phi, \Psi)\cong \mathrm{Mat}_{e\times d}(K\{\tau\})_{< N}.
\end{equation}
\noindent
\textbf{Step 3.} \smallskip\\
In the space $\mathrm{Mat}_{e\times d}(K\{\tau\})_{< N}$ we choose  the  set of generators of the following form:
$$E_{i\times j}c\tau^k\quad \textnormal{for}\quad c\in K,\quad\textnormal{and} \quad k=0,1,2,\dots, n_{i\times j}-1$$
\noindent
\textbf{Step 4.} \smallskip\\
 In order to  determine the structure of a \tm module on $\Ext^1_{\tau}(\Phi,\Psi)$
 we transfer the  $\F_q[t]-$module  structure from $\Ext^1_{\tau}(\Phi,\Psi)=\Der(\Phi,\Psi)/\Derin(\Phi,\Psi)$ to the space  $\mathrm{Mat}_{e\times d}(K\{\tau\})_{< N}$ via the isomorphism \eqref{eq:algorytm_reduckji_step2}.
 To achieve this it is enough to find the value of  multiplication by 
$t$  on the generators from Step 3. 
Recall that multiplication of $\delta\in \Der(\Phi,\Psi)/\Derin(\Phi,\Psi)$ by  $t$ is given by the following formula cf. \cite{pr}: 
$t*\delta = \Psi_t\cdot \delta.$
Thus in this step we determine  the values of multiplication by  $t$ on the generators chosen in Step 3:
$$t*E_{i\times j}c\tau^k=\Psi_t\cdot E_{i\times j}c\tau^k.$$
\textbf{Step 5.} \smallskip\\
If  $t*E_{i\times j}c\tau^k\notin \mathrm{Mat}_{e\times d}(K\{\tau\})_{< N}$ then we reduce the terms of too big degrees 
by means of the inner biderivations from Step 2. We  finish the process when the reduced biderivation is in $\mathrm{Mat}_{e\times d}(K\{\tau\})_{< N}$.\medskip\\
\textbf{Step 6.} \smallskip\\
We write the reduced values of $t*E_{i\times j}c\tau^k$ in the coordinate system $E_{i\times j}\tau^k$. 
\smallskip\\
\textbf{Step 7.} \smallskip\\
Since every expression of the form
 $c^{q^m}$ can be written as the evaluation of the monomial  $\tau^m$ at $\tau=c$ the coordinate vector for every value of  $t*E_{i\times j}c\tau^k$ can be written as the vector of polynomials in  $K\{\tau\}$ evaluated in  $\tau=c$. 
\medskip\\
\textbf{Final Step} \smallskip\\    
We form a matrix  $\Pi_t$, whose columns are the  vectors of polynomials computed in  Step 7. Matrix $\Pi_t$ 
determines the \tm module structure on $\Ext^1_{\tau}(\Phi,\Psi)$. \\
\begin{remark}
Notice that the condition for the degrees of  \tm modules is not sufficient for the reduction algorithm to be always executable. 
This follows from the fact that we cannot always perform
  Step 2 i.e. we cannot  find the  basis  $(U_{i\times j})_{i,j}$ such that  the above mentioned inner biderivations fulfill conditions A., B., C. and D. in Step 2.  The reader can see this by taking
	$$\Phi_t=\left[\begin{array}{cc}
		\theta+\tau^2 & \tau\\
		\tau & \theta
	\end{array}\right]\quad \textnormal{and}\quad 
	\Psi_t=\theta +\tau.$$ 
\end{remark}

However, we have the following:

\begin{prop}\label{prop:poprawnosc_algorytmu_redukcji}
	Let $\Phi$ and $\Psi$ be \tm modules such that  $\rk\Phi>\rk\Psi$ and for which there exists a basis 
	 $\big(U_{i\times j}\big)_{i,j}$ satisfying the conditions of  Step 2 of \tm reduction algorithm. 
	 Then the \tm reduction algorithm is correct i.e. we finally obtain the matrix 
	  $\Pi_t$ of the following form: 
	 $$\Pi_t=(\theta I+N_{\Pi})\tau^0+ \sum_{i=1}^{finite}C_i\tau^i,$$
	 where  $C_i\in \mathrm{Mat}(K)$ and $N$ is a nilpotent matrix.  
\end{prop}

\begin{proof}
	Let 
	$$\Phi=(\theta I+N_{\Phi})\tau^0+ \sum_{i=1}^{n}A_i\tau^i,\quad\textnormal{and} \quad \Psi=(\theta I+N_{\Psi})\tau^0 + \sum_{i=1}^{m}B_i\tau^i,$$
	where $n=\rk\Phi>\rk\Psi=m$. 
	Without loss of generality we may assume that 
	 $N_{\Psi}$ is a lower triangular matrix.

Assume there exists a basis 
		$\big(U_{i\times j}\big)_{i,j}$ of the space $\mathrm{Mat}_{e\times d}(K\{\tau\})$ fulfilling conditions  A., B., C. and D. of Step 2. Consider the following biderivations:
\begin{align}\label{eq:mnozenie_przez_t}
	t*E_{i \times j}c\tau^k&= \Psi_t\cdot E_{i\times j}c\tau^k=\Big((\theta I+N_{\Psi})\tau^0+ \sum_{l=1}^{m}B_l\tau^l\Big)\cdot E_{i\times j}c\tau^k=\\
	&=\Big(\theta I +N_{\Psi}\Big)E_{i\times j}c\tau^k+\sum_{l=1}^m c^{q^l} B_lE_{i\times j}\tau^{k+l} \nonumber
\end{align}
Step 2. guarantees that there exists a sequence of inner biderivations:
\begin{align*}
	\delta^{(c_1\tau^{k_1}U_{i_1\times j_1})},\quad \delta^{(c_2\tau^{k_2}U_{i_2\times j_2})},\quad  \delta^{(c_s\tau^{k_s}U_{i_s\times j_s})}.
\end{align*}
such that the reduced biderivation
\begin{equation}
	\label{eq:zredukowana_biderywacja}
	t*E_{i \times j}c\tau^k-\sum\limits_{l=1}^s\delta^{(c_l\tau^{k_l}U_{i_l\times j_l})}\in\mathrm{Mat}_{e\times d}(K\{\tau\})_{<N}.
\end{equation}
Notice that we do not reduce the coefficient at $\tau^k$ since 
 $E_{i \times j}c\tau^k\in\mathrm{Mat}_{e\times d}(K\{\tau\})_{<N}.$
First we will show that that the biderivation
 \eqref{eq:zredukowana_biderywacja} has the following form:
\begin{equation}\label{eq:postac_wieloianowa_biderywacji}
	\Bigg(\Big(\theta I +N_{\Psi}\Big)E_{i\times j}c+ \Big[ w_{a\times b, k}(c)\Big]_{a,b}\Bigg)\tau^k+\sum_{ l=0,\ l\neq k}^{\textnormal{finite}} \Big[ w_{a\times b, l}(c)\Big]_{a,b}\tau^{l},
\end{equation}
where  $w_{a\times b, l}(\tau)$ are twisted polynomials in  $K\{\tau\}$ with  coefficients at ${\tau}^0$ equal to zero.
Notice that some of these polynomials might be identically equal to zero.
The proof is by a finite induction on the number of  inner biderivations used in the reduction process.
 Notice that the biderivation (\ref{eq:mnozenie_przez_t}) i.e. before the start of the reduction process is of the form (\ref{eq:zredukowana_biderywacja}). Assume that,  after reduction by means of $r-1$ inner biderivations, the biderivation  \eqref{eq:mnozenie_przez_t} has the following form:
\begin{align*}
	t*E_{i \times j}c\tau^k-\sum\limits_{l=1}^{r-1}\delta^{(c_l\tau^{k_l}U_{i_l\times j_l})}&=
		\Bigg(\Big(\theta I +N_{\Psi}\Big)E_{i\times j}c+ \Big[ v_{a\times b, k}(c)\Big]_{a,b}\Bigg)\tau^k+\\
		&+\sum_{l=0,\ l\neq k}^{\textnormal{finite}} \Big[ v_{a\times b, l}(c)\Big]_{a,b}\tau^{l},
\end{align*}
for some twisted polynomials  $ v_{a\times b,l}(\tau)$ with no constant terms. 
We will show that after next reduction by the biderivation
 $\delta^{(c_r\tau^{k_r}U_{i_r\times j_r})}$ we still obtain a biderivation of the form  \eqref{eq:postac_wieloianowa_biderywacji}.     First we find the form of this inner biderivation.

\begin{align}\label{eq:biderywacja_wewnetrzna_r}
	&\delta^{(c_r\tau^{k_r}U_{a\times b})}=c_r\tau^{k_r}U_{i_r\times j_r}\Phi - \Psi c_r\tau^{k_r}U_{i_r\times j_r}\nonumber \\
	&=c_r\Big( \nadwzorem{(\theta^{(k_r)}-\theta)U_{i_r\times j_r}^{(k_r)} + U_{i_r\times j_r}^{(k_r)}N_{\Phi}-N_{\Psi}U_{i_r\times j_r}^{(k_r)}}{:=\big[d_{a\times b}\big]_{a,b}} \Big)\tau^{k_r}+\nonumber\\
	& + \sum_{l=1}^m \Big(c_r\podwzorem{U_{i_r\times j_r}^{(k_r)}A_l^{(k_r)}}{:=\big[\widehat{a}_{a\times b,l}\big]_{a,b}}- \podwzorem{B_lc_r^{(l)}U_{i_r\times j_r}}{:=\big[\widehat{b}_{a\times b,l}\big]_{a,b}}\Big)\tau^{k_r+l}+
	\sum_{l=m+1}^n c_r\podwzorem{U_{i_r\times j_r}^{(k_r)}A_l^{(k_r)}}{:=\big[\widehat{a}_{a\times b,l}\big]_{a,b}}\tau^{k_r+l} \\
	&=c_r\big[d_{a\times b}\big]_{a,b}\tau^{k_r}+ \sum_{l=1}^m \Big(c_r\big[\widehat{a}_{a\times b,l}\big]_{a,b}-c_r^{(l)}\big[\widehat{b}_{a\times b,l}\big]_{a,b}\Big)\tau^{k_r+l} \nonumber\\
	&+ \sum_{l=m+1}^n c_r\big[\widehat{a}_{a\times b,l}\big]_{a,b}\tau^{k_r+l}\nonumber
\end{align}
where  $A_*^{(k_r)}$ means that every entry of this matrix is raised to the power $q^{k_r}$.  
Recall that according to  Step 2  the inner biderivation $\delta^{(c_r\tau^{k_r}U_{i_r\times j_r})}$  as an  $i_r\times j_r$ entry has a twisted polynomial with the coefficient at the biggest power $\tau^M$ equal to $c_r$. This biderivation is used for reduction of the term   $v_{i_r\times j_r,M}(c)$ by putting $c_r=\dfrac{v_{i_r\times j_r,M}(c)}{\widehat{a}_{i_r\times j_r}}$. 
Write the inner biderivation $\delta^{(c_r\tau^{k_r}U_{a\times b})}$in the following form:
\begin{align*}
	\delta^{(c_r\tau^{k_r}U_{a\times b})}&=
	\sum_{l=k_r}^{k_r+n}\big[u_{a\times b,l}(c_r)\big]_{a,b}\tau^{l},
\end{align*}
where $u_{a\times b,l}(\tau)$ are twisted polynomials in $K\{\tau\}$ with no constant terms.
Then 
\begin{align*}
	&t*E_{i \times j}c\tau^k-\sum\limits_{l=1}^{r}\delta^{(c_l\tau^{k_l}U_{i_l\times j_l})}=
	\Bigg(\Big(\theta I +N_{\Psi}\Big)E_{i\times j}c \\
	&+ \Big[ v_{a\times b, k}(c)-u_{a\times b,k}(c_r)\Big]_{a,b}\Bigg)\tau^k
	+\sum_{l=0,\ l\neq k}^{\textnormal{finite}} \Big[ v_{a\times b, k+l}(c)- u_{a\times b,l}(c_r)\Big]_{a,b}\tau^{l}, \\
\end{align*}
where we set $u_{a\times b,l}(\tau)=0$ for $l\notin\{k_r,k_r+1,\dots, k_r+n\}$. 

Substituting
$$w_{a\times b,l}(\tau)=v_{a\times b,l}(\tau)-u_{a\times b,l}\Bigg( \dfrac{v_{i_r\times j_r,M}(\tau)}{\widehat{a}_{i_r\times j_r}}\Bigg)$$
we see that the twisted polynomials $w_{a\times b,l}(\tau)$  have no constant terms since  $v_{a\times b, l}(\tau)$ and $u_{a\times b, l}(\tau)$  have no constant terms. This finishes the inductive step.

The form of the biderivation \eqref{eq:postac_wieloianowa_biderywacji} implies that one can perform  Step 7 and the Final Step. 
In this way we obtain a matrix
 $\Pi_t$ of the following form:
$$\Pi_t=(\theta I+N_\Pi)\tau^0+\sum_{l=1}^{\textnormal{finite}}C_l\tau^l.$$
It remains to show that
 $N_\Pi$ is a nilpotent matrix. Since the twisted polynomials $w_{a\times b,k+l}(\tau)$ in biderivations \eqref{eq:postac_wieloianowa_biderywacji}  have no constant terms then  the entries of  $N_\Pi$ for the multiplication $t*E_{i\times j}c\tau^k$ 
come from the entries of $N_{\Psi}E_{i\times j}$.  Since $N_{\Psi}$ is a lower triangular  matrix then  $N_\Pi$ 
is also  lower triangular  and therefore nilpotent. 
\end{proof}
\begin{definition}
If the \tm module structure  on $\Ext^1_{\tau}(\Phi,\Psi)$ can be obtained from the described above  \tm reduction algorithm we will say that such a  \tm module structure  comes from  \tm reduction. 
\end{definition}

\begin{remark} 
In this notation matrices $U_{i\times j}$ for some cases  considered in \cite{kk04} had the following forms:
\begin{itemize}
\item{} $U_{i\times j}=E_{i\times j}A_n^{-1}$ for ${\Ext}_{\tau}^1(\Phi, C^{\otimes e})$ where $\Phi_t=(\theta I+ N_{\Phi})\tau^0+\sum\limits_{j=1}^n A_j\tau^j,$ $n>1,$ is a \tm module of dimension $d$, such 
	that $A_n$ is an invertible matrix and    $C^{\otimes e}$ is the $e-$th tensor of the Carlitz module.
	$E_{i\times j}\in  M_{e\times d}(K)$ denotes here the  matrix that has $1$ at the place $(i,j)$ and $0$ otherwise.
\item{}$U_{i\times j}=E_{j}A_n^{-1}$	 for ${\Ext}_{\tau}^1(\Phi, \psi)$ where $\Phi_t=(\theta I+ N_{\Phi})\tau^0+\sum\limits_{j=1}^n A_j\tau^j $ is a \tm module of dimension $d$, where $A_n$ is an invertible matrix and   $\psi_t=\theta +\sum\limits_{j=1}^mb_j\tau^j$ is a Drinfeld module satisfying the condition $\rk \Phi>\rk\psi.$ 
$E_i\in M_{1\times d}(K)$ denotes here a row matrix with $1$ on the $i$-th place and $0$ otherwise.
\end{itemize}
The reader can easily find the forms of $U_{i\times j}$ for all other cases i.e. for ${\Ext}^1_{\tau}(\phi,\psi)$ where $\phi$ and $\psi$ are Drinfeld modules with $\rk\phi>\rk\psi,$\\
 $\Ext^1_{\tau}\Big(\prod_{i=1}^n\phi_i, \prod_{j=1}^m\psi_j\Big)$ where $\rk\phi_i>\rk\psi_j$ for $i=1,\dots,n,$ $j=1,\dots, m$ and $\Ext^1_{\tau}\Big(\Phi, \prod_{j=1}^m\psi_j\Big)$ where $\Phi$ is a \tm module, $\psi_j$ are Drinfeld modules and 
 $\rk\Phi>\rk\psi_j$, $j=1,\dots,m.$ 
 
\end{remark}

Notice that by \cite[ Proposition 5.1]{kk04},  \cite[ Theorem 6.1]{kk04}, \cite[ Theorem 8.3]{kk04}, \cite[Theorem 8.4]{kk04},
and \cite[Theorem 9.2]{kk04} all the above mentioned \tm module structures come from \tm reduction.

For the  \tm module structures coming from \tm reduction we have the following useful lemma:

\begin{lem}\label{lem:algorytm_reduckji_dla_krotkiego_ciagu}\label{sec3}
	Let 
	\begin{equation}\label{exact}
	 0\lra \Psi\lra \Upsilon \lra \Phi\lra 0
	 \end{equation}
	 be a short exact sequence of \tm modules.
	\begin{itemize}
		\item[$(i)$] If $\zeta$ is a \tm module such that $\Hom_{\tau}(\Psi, \zeta)=0$ and
		the spaces $\Ext^1_{\tau}(\Psi, \zeta)$ and $\Ext^1_{\tau}(\Phi, \zeta)$ have  \tm module structures coming from
		 \tm reduction then 
		 $\Ext^1_{\tau}(\Upsilon, \zeta)$ has also a \tm module structure coming from \tm reduction. 
		\item[$(i^D)$] If $\zeta$ is a \tm module such that  $\Hom_{\tau}(\zeta,\Phi)=0$ and the spaces
		 $\Ext^1_{\tau}(\zeta, \Psi)$ and $\Ext^1_{\tau}(\zeta, \Phi)$ have  \tm module structures coming from
		 \tm reduction then 
		 $\Ext^1_{\tau}(\zeta, \Upsilon)$ has also a \tm module structure coming from \tm reduction. 
	\end{itemize} 
\end{lem}
\begin{proof}
	We will prove part $(i)$. The proof of part $(i^D)$ is analogous.
	
	Denote $\dim\Phi=d$, $\dim\Psi=e$ and $\dim\zeta=f$.
	Assume that the sequence  (\ref{exact}) is given by a biderivation $\delta$, i.e. 
	$$\Upsilon= \left[\begin{array}{c|c}
		\Phi & 0\\ \hline
		\delta& \Psi
	\end{array}\right].$$ 
	Applying the functor $\Hom_{\tau}(-,\zeta)$ to the sequence  (\ref{exact}), by Theorem \ref{thm:long_sequence}, we obtain  the following exact sequence:
	\begin{align*}
		0&\lra \Hom_{\tau}(\Phi,\zeta)\lra \Hom_{\tau}(\Upsilon,\zeta) \lra \nadwzorem{\Hom_{\tau}(\Psi,\zeta)}{=0}\lra \\&
		\lra \Ext^1_{\tau}(\Phi,\zeta)\lra \Ext^1_{\tau}(\Upsilon,\zeta) \lra \Ext^1_{\tau}(\Psi,\zeta)\lra 0.
	\end{align*}
	and therefore the following short exact sequence: 
	$$0\lra \Ext^1_{\tau}(\Phi,\zeta)\lra \Ext^1_{\tau}(\Upsilon,\zeta) \lra \Ext^1_{\tau}(\Psi,\zeta)\lra 0,$$
	where both the left an right terms are  \tm modules. 
	Every element from
	  $\Ext^1_{\tau}(\Upsilon,\zeta)\cong \Der(\Upsilon,\zeta)/\Derin(\Upsilon,\zeta)$ is given by the following biderivation: 
	$$\left[\begin{array}{c|c}
		\delta_1 &\delta_2
	\end{array}\right]\in  \mathrm{Mat}_{f \times d+e }(K\{\tau\})= \Der(\Phi, \zeta)\times \Der(\Psi, \zeta).$$
	
	Since the \tm module structures   on $\Ext^1_{\tau}(\Phi,\zeta)$ and $\Ext^1_{\tau}(\Phi,\zeta)$ come from \tm reduction there exist bases   
	$\Big(U^{\Phi,\zeta}_{i\times j}\Big)$ for $i\in\{1,\dots, f\}$ and $j\in\{1,\dots, d\}$ 
	in the space $\mathrm{Mat}_{f\times d}(K)$ and 
	$\Big(U^{\Psi,\zeta}_{i\times l}\Big)$ for $i\in\{1,\dots, f\}$ and $l\in\{1,\dots, e\}$
	in the space  $\mathrm{Mat}_{f\times e}(K)$ which satisfy the conditions in  Step 2. 
	Consider the following basis of the space $\mathrm{Mat}_{f \times d+e }(K)$
		$$\Bigg(\left[\begin{array}{c|c}
			U^{\Phi,\zeta}_{i\times j}&0
		\end{array}\right], \left[\begin{array}{c|c}
		0&U^{\Psi,\zeta}_{i\times l}
	\end{array}\right]\Bigg),$$
where $i\in\{1,\dots, f\}$,
	$j\in\{1,\dots, d\}$ and $l\in\{1,\dots, e\}.$ 
	Then an inner biderivation  from $\Derin(\Upsilon,\zeta)$ is built from the two biderivations  of the following  forms:
	\begin{itemize}
	\item[(a)]
	\begin{align*}
		\delta^{\Big(c\tau^k\left[\begin{array}{c|c}
				U^{\Phi,\zeta}_{i\times j} &0
			\end{array}\right]\Big)} &= c\tau^k\left[\begin{array}{c|c}
		U^{\Phi,\zeta}_{i\times j} &0
	\end{array}\right]\cdot \left[\begin{array}{c|c}
			\Phi & 0\\ \hline
			\delta& \Psi
		\end{array}\right] - \zeta c\tau^k\left[\begin{array}{c|c}
		U^{\Phi,\zeta}_{i\times j} &0
	\end{array}\right] \\
	&=\left[\begin{array}{c|c}
			c\tau^kU^{\Phi,\zeta}_{i\times j}\Phi-\zeta c\tau^kU^{\Phi,\zeta}_{i\times j} &0
		\end{array}\right]=\left[\begin{array}{c|c}
		 \delta^{(c\tau^kU^{\Phi,\zeta}_{i\times j})}&0
	\end{array}\right]
	\end{align*}
	where $\delta^{(c\tau^kU^{\Phi,\zeta}_{i\times j})}\in\Derin(\Phi, \zeta)$.
	\item[(b)]
	\begin{align*}
		\delta^{\Big(c\tau^k\left[\begin{array}{c|c}
				0&U^{\Psi,\zeta}_{i\times l}
			\end{array}\right] \Big)} &= \left[\begin{array}{c|c}
			0&c\tau^kU^{\Psi,\zeta}_{i\times l}
		\end{array}\right]\cdot \left[\begin{array}{c|c}
			\Phi & 0\\ \hline
			\delta& \Psi
		\end{array}\right] - \zeta \left[\begin{array}{c|c}
			0&c\tau^kU^{\Psi,\zeta}_{i\times l} 
		\end{array}\right]\\
	&=\left[\begin{array}{c|c}
			c\tau^kU^{\Psi,\zeta}_{i\times l}\cdot\delta & \delta^{(c\tau^kU^{\Psi,\zeta}_{i\times l})}
		\end{array}\right]
	\end{align*}
	where $\delta^{(c\tau^kU^{\Psi,\zeta}_{i\times l})}\in\Derin(\Phi, \zeta)$.
	\end{itemize}
	Both of these biderivations fulfill conditions  A., B., C. and  D. of Step 2. 
	Then Proposition \ref{prop:poprawnosc_algorytmu_redukcji} yields the assertion (i).	
\end{proof}

The following interesting question was asked by the referee:
When the space
$\Ext^1$ can be expressed  as a quotient of the $\Hom$-space? Such a situation will take place when the last two   $\Ext-$groups in the six-term exact sequence  described in Theorem \ref{thm:long_sequence} are isomorphic. The following lemma yields the equivalent condition for this in the language of biderivations.

\begin{lem}\label{lem:iso_extow}
	Let 
	\begin{equation*}
	 \delta:\quad 0\lra \Psi\uplra{i} \Upsilon \uplra{\pi} \Phi\lra 0
	 \end{equation*}
	 be a short exact sequence of \tm modules given by the biderivation $\delta$, where $\dim\Psi=e$ and $\dim\Phi=d$. Let $\zeta$ be a \tm module of dimension $f$. Then 
	\begin{itemize}
		\item[$(i)$] an ${\mathbb F}_q[t]$-homomorphism  $(-\pi)\circ-:\Ext^1_{\tau}(\zeta, \Upsilon)\uplra{\cong} \Ext^1_{\tau}(\zeta, \Phi)$ is an  isomorphism if and only if for arbitrary matrices
        $$\delta_2\in \mathrm{Mat}_{e\times f}\big(K\{\tau\}\big)\quad \textnormal{and}\quad u\in \mathrm{Mat}_{d\times f}\big(K\{\tau\}\big)$$  
        there exists a morphism $F\in\Hom_\tau(\zeta,\Phi)$ such that
        \begin{equation}\label{war1}
        \delta_2+\delta \big(F-u\big)\in \Derin(\zeta, \Psi).
        \end{equation}
		\item[$(i^D)$] an ${\mathbb F}_q[t]$-homomorphism  $-\circ(-i):\Ext^1_\tau(\Upsilon, \zeta)\uplra{\cong} \Ext^1_{\tau}(\Psi, \zeta)$ is an  isomorphism if and only if for arbitrary matrices
        $$\delta_1\in \mathrm{Mat}_{f\times d}\big(K\{\tau\}\big)\quad \textnormal{and}\quad  u\in \mathrm{Mat}_{f\times e}\big(K\{\tau\}\big)$$  
         there exists a morphism $F\in\Hom_\tau(\Psi, \zeta)$ such that 
        \begin{equation}\label{war2}
        \delta_1- \big(F-u\big)\delta\in \Derin(\Phi, \zeta).
        \end{equation}
	\end{itemize} 
\end{lem}

\begin{proof}
    We will prove the lemma for the case $(i)$. The proof for the second case is dual. Recall that the map  $i$ in the sequence $\delta$ is an inclusion on the second coordinate and the map  $\pi$ is a projection on the first coordinate.
    From \eqref{iso_ext} we have 
    $\Ext^1_{\tau}(\zeta, \Phi)\cong \Der(\zeta, \Phi)/\Derin(\zeta, \Phi)$ and an isomorphism $\Der(\zeta, \Phi)\cong \mathrm{Mat}_{d\times f}\big(K\{\tau\}\big)$. 
    Similarly, $\Ext^1_{\tau}(\zeta, \Upsilon)\cong \Der(\zeta, \Upsilon)/\Derin(\zeta, \Upsilon),$ where $\Der(\zeta, \Upsilon)\cong \mathrm{Mat}_{e+d\times f}\big(K\{\tau\}\big)=\Big[ \mathrm{Mat}_{e\times f}\big(K\{\tau\}\big),\mathrm{Mat}_{d\times f}\big(K\{\tau\}\big)\Big]^T$. 
    Then the map $(-\pi)\circ-:\Der(\zeta, \Upsilon)/\Derin(\zeta, \Upsilon)\lra \Der(\zeta, \Phi)/\Derin(\zeta, \Phi)$ on the biderivation level is given by the following formula:
    $$\left[\begin{array}{c}
        \delta_1 \\
         \delta_2
    \end{array}\right]+\Derin(\zeta, \Upsilon)\longmapsto -\delta_1+\Derin(\zeta, \Phi).$$
    This map will be an isomorphism if and only if its kernel is trivial i.e. the condition $-\delta_1\in \Derin(\zeta, \Phi)$  implies that $\left[\begin{array}{c}
        \delta_1 \\
         \delta_2
    \end{array}\right]\in \Derin(\zeta, \Upsilon)$. 
   This is equivalent to the statement that for every matrix
     $u\in\mathrm{Mat}_{d\times f}\big(K\{\tau\}\big)$, where $-\delta_1=u\zeta-\Phi u$ there exists a  matrix $\left[\begin{array}{c}
        w_1 \\
         w_2
    \end{array}\right]\in \mathrm{Mat}_{e+d\times f}\big(K\{\tau\}\big)$ such that 
    \begin{align*}
        \left[\begin{array}{c}
        \delta_1 \\
         \delta_2
    \end{array}\right]&=
    \left[\begin{array}{c}
        w_1 \\
         w_2
    \end{array}\right]\zeta - \Upsilon  \left[\begin{array}{c}
        w_1 \\
         w_2
    \end{array}\right]=
     \left[\begin{array}{c}
        w_1\zeta-\Phi w_1\\
         w_2\zeta-\Psi w_2-\delta w_1
    \end{array}\right].
    \end{align*}
    The equality $\delta_1=w_1\zeta-\Phi w_1$ is equivalent to the fact that the matrix  $F=u+w_1$ is a morphism of  \tm modules  $\zeta \rightarrow \Phi$. Moreover, the equality  $\delta_2 =w_2\zeta-\Psi w_2-\delta w_1$ is equivalent to 
    \begin{align*}
      \delta_2+ \delta \big(F-u\big)=  \delta_2+ \delta w_1\in \Derin(\zeta, \Psi).
    \end{align*}
     
\end{proof}
\begin{remark}
 The conditions (\ref{war1}) and (\ref{war2}) are difficult to handle because they require understanding of the morphism space between two \tm modules. In general this is a very difficult task. For Drinfeld modules nice, but by no means easy, algorithm of calculating the space of morphisms was developed in  \cite{kp22}.  
\end{remark}
It would be nice to have an example for which one of the conditions \eqref{war1} or  \eqref{war2} is fulfilled.

In the following propositions the answers are negative. 
\begin{prop}
Let $\Phi$ and $\Psi$ be \tm modules with the zero nilpotent matrices $N_{\Phi}$ and $N_{\Psi}$.  Let  
 \begin{equation*}
	 \delta:\quad 0\lra \Psi\uplra{i} \Upsilon \uplra{\pi} \Phi\lra 0
	 \end{equation*}
     be a short exact sequence of \tm modules given by the biderivation $\delta.$ If $\delta$ has no constant term then both conditions \eqref{war1} and \eqref{war2} are not fulfilled.
\end{prop}
\begin{proof}
Notice that if $N_{\Phi}=N_{\Psi}=0 $ then every inner biderivation from $\Der_{in}(\Phi,\Psi)$ has no constant term. Indeed, if $\dim\Phi=d$ and 
$\dim\Psi=e$ then for $V\in {\mathrm{Mat}}_{e\times d}(K\{\tau\})$ we have: 
\begin{align*}
{\delta^{(V)}}=&V\Phi-\Psi V=V\Big(\theta I+\sum A_i\tau^i\Big)-\Big(\theta I+\sum B_i\tau^i\Big)\\
=&\big(V\theta-\theta V\big)+V\sum A_i\tau^i-\sum B_i\tau^iV
\end{align*}
and  ${\partial}{\delta}^{(V)}={\partial}(V\theta-\theta V)=\partial V\theta -\theta \partial V=0$, where $\partial$ is a differentiation given by the following formula:
$\partial \Big(V_0+\sum V_i\tau^i\Big)=V_0.$
We will show the assertion of the proposition 
for the condition \eqref{war1}. 
Consider the equality
${\delta}_2+{\delta}(F-u)={\delta}^{(V)}$
for some $V\in {\mathrm{Mat}}_{e\times d}(K\{\tau\}).$  Differentiating both sides we see that $\partial{\delta}_2=0$ which contradicts the fact that ${\delta}_2$ can be arbitrary.
\end{proof}
\begin{prop}
Let $\phi$, $\psi$ and $\zeta$ be Drinfeld modules defined over an algebraically closed field $K$ and 
\begin{equation*}
	 \delta:\quad 0\lra \psi\uplra{i} \Upsilon \uplra{\pi} \phi\lra 0
	 \end{equation*}
     be a short exact sequence 
 given by a biderivation $\delta$ such that $\partial\delta\neq 0$. Then both conditions \eqref{war1} and \eqref{war2} are not fulfilled. 
\end{prop} 
\begin{proof}
We will give a proof for the condition \eqref{war1}. The proof for \eqref{war2} is dual.
    
    We will act on the sequence $\delta$ by means of the 
    functor
      $\Hom_\tau (\zeta, -)$.  
     In the condition \eqref{war1}  put  $u=0$. If the condition \eqref{war1} holds true,  then there exists a morphism
      $F\in\Hom_\tau(\zeta, \phi)$ such that 
     $$\delta_2+\delta F \in \Derin(\zeta,\psi).$$ 
     We will show that such a morphism cannot exist for arbitrary 
      $\delta_2$. Notice that every inner biderivation in $\Derin(\zeta,\psi)$ has no constant term. Therefore differentiating  the above equality we obtain
     $$\partial\delta_2+\partial F\cdot \partial \delta =0,$$
     where $\partial\delta_2,\partial F$ and $\partial \delta$ are constant terms of  $\delta_2, F$ and $\delta$ respectively. Since $\partial\delta_2$ can be an arbitrary element of the field $K$ we  see that $\partial F =-\partial\delta_2/\partial \delta$ is also arbitrary. This means that the constant term of the morphism between Drinfeld modules can be equal to an arbitrary element of $K$. We will show that this impossible if $K$ is algebraically closed.  If $\rk\zeta \neq \rk \psi$, this follows from the fact that $\Hom_\tau(\zeta,\phi)=0$. So assume that  $\rk\zeta=\rk\phi$ and 
     $$\zeta= \theta+\sum_{i=1}^na_i\tau^i,\quad 
     \phi=\theta+\sum_{i=1}^nb_i\tau^i
     \quad \textnormal{and}\quad F=\sum_{i=1}^mf_i\tau^i\in \Hom_\tau(\zeta,\phi).$$ 
     Comparing the coefficients at
      $\tau^i$ in the equality $F\zeta=\phi F$ we obtain the following equalities:
     \begin{align}\label{eq:zaleznosc_morfizm}
     f_k\Big(\theta^{(k)}-\theta \Big) =\sum_{i=1}^k \Big(b_if_{k-i}^{(i)} - f_{k-i}a_i^{(k-i)} \Big)\quad \textnormal{for}\quad k=0,1,\dots n+m, 
     \end{align}
     where we set $f_j=0$ for $j>m$ and $a_l=b_l=0$ for $l>n$.

     We consider two cases.
     If $\theta^{(s)}\neq \theta$ for every $s=1,2,\dots, m$ then from the equations  \eqref{eq:zaleznosc_morfizm} for  $k=1,2,\dots, m$ it follows that every term  $f_1$, $f_2$, \dots $f_m$ is uniquely determined by the values of some polynomial at $f_0$ . Let $f_m=w(f_0)$ for  $w(x)\in K[x]$. Consider $g(x)=w(x)-\mu$ for fixed  $\mu\in K$. Since $K$ is algebraically closed there exists $f_0\in K$ such that $g(f_0)=0$ i.e.  $f_m=w(f_0)=\mu$. This means that the coefficient at the highest exponent of $\tau$ can be arbitrary. But the equality \eqref{eq:zaleznosc_morfizm} for $k=m+n$ implies that $f_m$ must fulfill the following relation
     $$b_nf_m^{(n)}=f_ma_n^{(m)}.$$
     Thus, $f_m$ is a root of a polynomial $h(x)=b_nx^{q^n}-a_n^{q^m}x$ and therefore cannot be an arbitrary element of  $K$.

Now let $\theta^{(s)}=\theta$ for some $s\leq m$. Then simiarly as before from equalities \eqref{eq:zaleznosc_morfizm} for $k=1,2,\dots, s-1$ we obtain that $f_1,$ $f_2$, $\dots, f_{s-1}$ are uniquely determined by $f_0$ by means of some  polynomials evaluated  at $f_0.$ 
     Let $f_1=w_1(f_0)$, $\dots, f_{s-1}=w_{s-1}(f_0)$. 
     Then \eqref{eq:zaleznosc_morfizm} for $k=s$ implies the following equality:
     \begin{align*}
     0=\sum_{i=1}^s \Big(b_if_{s-i}^{(i)} - f_{s-i}a_i^{(k-i)} \Big)= 
     \sum_{i=1}^s \Big(b_iw_{s-i}(f_0)^{(i)} - w_{s-i}(f_0)a_i^{(k-i)} \Big).
     \end{align*}
     Therefore  $f_0$ is a root of the above polynomial. This contradicts the fact that  $f_0$ can be arbitrary. 
     
So, the condition \eqref{war1}  cannot be fulfilled for any 
      $\delta_2$. 
    
\end{proof}

\section{\tm  reduction algorithm in case of two \tm modules}\label{sec4}

In this section we present an algorithm for determination of a \tm module structure on the space 
$\Ext^1(\Phi,\Psi)$, where $\Phi$ and $\Psi$ are \tm modules such that $\rk\Phi>\rk\Psi$ and the matrix at  $\tau^{\rk\Phi}$ in  ${\Phi}_t$ is invertible.  As a result we obtain the following theorem which is a generalization of Theorems 8.4 and 9.2 of 
 \cite{kk04}.
\begin{thm}\label{thmm} 
	Let $\Phi=(\theta I+N_{\Phi})\tau^0+\sum\limits_{i=1}^{\rk\Phi}A_i\tau^i$ resp.  
	$\Psi=(\theta I+N_{\Psi})\tau^0+\sum\limits_{i=1}^{\rk\Psi}B_i\tau^i$ be  \tm modules of dimensions $d$ resp. $e.$ If the matrix  $A_{\rk\Phi}$ is invertible and $\rk\Phi>\rk\Psi$ then  
		\begin{itemize}
			\item[$(i)$] $\Ext^1_{\tau}(\Phi,\Psi)$ has a natural structure of a  \tm module coming from \\
			 \tm reduction,
			\item[$(ii)$] there exists a short exact sequence of  \tm modules
			$$0\lra \Ext_{0,\tau}(\Phi,\Psi)\lra \Ext^1_{\tau}(\Phi, \Psi )\lra \G_a^{s}\lra0,$$
			where $s$   
			is the number of pairs $(i,j)$ for which the matrices
		$E_{i\times j}A_n^{-1}N_{\Phi}$ and $N_{\Psi} E_{i\times j}A_n^{-1}$ are the same.
		\end{itemize}

\end{thm}	
\begin{proof} 
Part (i) follows from Proposition \ref{prop:poprawnosc_algorytmu_redukcji} by taking $U_{i\times j}=E_{i\times j}\cdot 
A_n^{-1}.$
Notice that reasoning analogous to that in 
 Theorem 9.2 of \cite{kk04}  shows that $s$  is equal to the number of inner biderivations  $\delta^{(c\tau^0U_{ i \times j})}$ that are  in $\Der_0(\Phi,\Psi)$. Recall that  $\delta^{(c\tau^0U_{ i \times j})}\in \Der_0(\Phi,\Psi)$ iff $\partial\delta^{(c\tau^0U_{ i \times j})}_t=0$. Therefore to finish the proof, it is enough to notice that the constant term of the biderivation $\delta^{(c\tau^0U_{ i \times j})}_t $ 
\begin{align*}
	c\Big( U_{ i \times j}N_{\Phi}-N_{\Psi}U_{ i \times j} \Big)&=
	c\Big(E_{i\times j}\cdot 
	A_n^{-1}N_{\Phi}-N_{\Psi}E_{i\times j}\cdot 
	A_n^{-1} \Big)= 0
\end{align*}	
iff the matrices $E_{i\times j}A_n^{-1}N_{\Phi}$ and $N_{\Psi} E_{i\times j}A_n^{-1}$ are the same.
\end{proof}	
The following example illustrates Theorem \ref{thmm}
\begin{example}
 Let $$\Phi_t=\left[\begin{array}{cc}
		\theta & {\tau}^3\\
		1+{\tau}^3 & \theta
	\end{array}\right]\quad \textnormal{and}\quad 
	\Psi_t=\left[\begin{array}{cc}\theta+\tau^2 & 0\\
		1 & \theta+\tau
	\end{array}\right].$$    
Then $N_{\Phi}=N_{\Psi}=\left[\begin{array}{cc}
		0 & 0\\
		1 & 0
	\end{array}\right],\quad \textnormal{and} \quad A_3^{-1}=\left[\begin{array}{cc}
		0 & 1\\
		1 & 0
	\end{array}\right].$ One readily verifies that
$E_{i\times j}A_3^{-1}N_{\Phi}=N_{\Psi} E_{i\times j}A_3^{-1}$ only for $(i,j)=(2,2). $ Thus $s=1.$ Accordingly, the reduction algorithm yields
$$\Ext^1_{\tau}(\Phi,\Psi)=\left[\begin{array}{cccccccccccc}
		\theta & -{\tau}^2 & 0 & 0 & 0 & 0 & 0 & 0 & 0 & 0 & 0 & 0 \\
0 &\theta & -{\tau}^2 & 0 & 0 & (\theta -{\theta}^{(1)}){\tau}^2 & 0 & 0 & 0 & 0 & 0 & 0 \\	{\tau}^2 & 0 & \theta+{\tau}^6 & 0 & {\tau}^4 & 0 & 0 & 0 & 0 & 0 &0 & 0 \\
0 & 0 & 0 & \theta & 0 & 0 & 0 & 0 & 0 & 0 &0 & 0 \\
0 & 0 & (\theta -{\theta}^{(1)}){\tau}^2 & 0 & \theta & 0 & 0 & 0 & 0 & 0 &0 & 0 \\
0 & {\tau}^4 & 0 & {\tau}^2 & 0 & \theta & 0 & 0 & 0 & 0 &0 & 0 \\
1 & 0 & {\tau}^4 & 0 & {\tau}^2 & 0 & \theta & 0 & -{\tau} & 0 & 0 &0  \\
0 & 1 & 0 & 0 & 0 & {\tau}^2 & \tau & \theta & 0 & 0 & 0 & {\tau}^2 \\
0 & 0 & 1 & 0 & 0 & 0 & 0 & \tau  & \theta & 0 & 0 & 0  \\
0 & {\tau}^2 & 0 & 1 & 0 & 0 & 0 & 0 & 0  & \theta & 0 & 0  \\
0 & 0 & {\tau}^2 & 0 & 1 & 0 & 0 & 0 & {\tau}^2 & \tau  & \theta & 0   \\
0 & 0 & 0 & 0 & 0 & 1 & 0 & 0 & 0 & 0   & \tau & \theta   \\
    \end{array}\right]$$
    The matrix for $\Ext_{0,\tau}(\Phi ,\Psi)$ comes then from the above matrix for $\Ext^1_{\tau}(\Phi,\Psi)$ by removing the fourth row and fourth column.
\end{example}

\subsection*{Notation}

In our pseudo-codes we use the following notation:

\noindent
$E_{i\times j}$ - $n \times m$ matrix with 1 in i-th row and j-th column
 and $0$ everywhere else \\
$\textsc{Deg(W)}$ - degree of  a $\tau$- polynomial $W$ \\
$\textsc{Dim(F)}$ - dimension of a t-module $F$ \\
$\textsc{Rows(M)}$ - number of rows of a matrix $M$ \\
$\textsc{Cols(M)}$ - number of columns of a matrix $M$ \\
$\textsc{Inverse(M)}$ - inverse of a matrix $M$ \\
$\textsc{Coefficient(w, n)}$ - coefficient $a_{\textproc{n}}$ of a
$K\{\tau\}$ polynomial $\textproc{w} = a_m {\tau}^m + ... + a_0$ \\
$\textproc{Substitute(expression, s, t)}$ - 
replaces symbol $\textproc{s}$ in $\textproc{expression}$
 with symbol $\textproc{t}$ (sometimes we substitute for any expression, then $\textproc{(\_)}$ stands for any expression
 cf. line 29 of Algorithm 5. )\\
 Notice that since we work with the non-commutative ring $K\{\tau\}$ certain care is needed in this substitution (cf. Step 6 and Step 7,  see also Example 4.1 of \cite{kk04}).\\
$\textproc{M}[i,j]$ - element $M_{i,j}$ of a matrix $\textproc{M}$ \\
$\begin{bmatrix}
    \boldsymbol{V_{1}} & ... & \boldsymbol{V_{k}}
    \end{bmatrix}$ - matrix with columns $\boldsymbol{V_{i}}$
\begin{algorithm}[H]
\caption{}\label{alg:helper1}
\begin{algorithmic}[1]
\Function{Pmult}{$f$, $g$}
\Comment{Multiplication of $\tau$ polynomials}

\Input
\Desc{f}{$\tau$ polynomial}
\Desc{g}{$\tau$ polynomial}
\EndInput
\Output
\Desc{$h$}{$\tau$ polynomial equal to $f \cdot g$}
\EndOutput

\State $n \gets \Call{Deg}{f}$
\State $m \gets \Call{Deg}{g}$
\State $z$ : array with $n + m + 1$ elements, initialized with $0$'s
\For {$i = 0..n$}
    \For {$j = 0..m$}
        \State $z[i+j] \gets
        \Call{Coefficient}{{f}, {i}} \cdot \Call{Coefficient}{{g}, {j}}^{(i)}$
    \EndFor
\EndFor
\State \Return $\sum_{k=0}^{m+n} z[k] \tau^k$
\EndFunction
\\
\Function{TMult}{$\Phi, \Psi$}
\Comment{Multiplication of matrices of $\tau$ polynomials}

\Input
\Desc{$\Phi$}{matrix of $\tau$ polynomials}
\Desc{$\Psi$}{matrix of $\tau$ polynomials}
\EndInput
\Output
\Desc{$\boldsymbol{X}$}{matrix of $\tau$ polynomials equal to $\Phi \cdot \Psi$}
\EndOutput

\State $n \gets \Call{Rows}{\Psi}$
\State $m \gets \Call{Cols}{\Psi}$
\State $p \gets \Call{Cols}{\Phi}$
\State $\boldsymbol{X}$ : $n \times p$ matrix
\For {$i = 1..n$}
    \For {$j = 1..p$}
        \State $sum \gets 0$
        \For {$k = 1..m$}
       \State $sum \gets sum + \Call{Pmult}{{\Psi[i,k]}, {\Phi[k,j]}}$
        \EndFor
        \State $\boldsymbol{X}[i,j] \gets sum$
    \EndFor
\EndFor
\State \Return $\boldsymbol{X}$
\EndFunction

\end{algorithmic}
\end{algorithm}

\begin{algorithm}[H]
\caption{}\label{alg:helper2}
\begin{algorithmic}[1]

\Function{Reduce1}{${\boldsymbol{V}}, {\Phi}, {\Psi}, {\boldsymbol{A_{n}^{-1}}}$}
\Comment{Reduction of t-module using \\ \qquad\qquad\qquad\qquad\quad\qquad\qquad\qquad\qquad\quad\,\, inverse matrix}

\Input
\Desc{$\boldsymbol{V}$}{module to be reduced}
\Desc{$\Phi$} { strictly pure t-module }
\Desc{$\Psi$}{ a t-module of degree less than degree $\Phi$}
\Desc{$\boldsymbol{A_{n}^{-1}}$}{inverse of the leading matrix of $\Phi$}
\EndInput

\Output
\Desc{$\boldsymbol{V}$}{reduced module}
\EndOutput

\State $n \gets \Call{Rows}{\boldsymbol{V}}$
\State $m \gets \Call{Cols}{\boldsymbol{V}}$
\State $r \gets \Call{Deg}{\Phi}$

  \Comment{we reduce degree of $\tau$ polynomial at position\\ \qquad\qquad\qquad\qquad\qquad\quad\,\, $i$, $j$ in 
        $\boldsymbol{V}$ to degree less than $r$}
\For {$i = 1..n$}
    \For {$j = 1..m$}
      
        \State $r' \gets \Call{Deg}{\boldsymbol{V}[i,j]}$

        \While {$ r' \geq r $}
            \State $a \gets \Call{Coefficient}{\boldsymbol{V}[i,j], r'} $
            \State $\boldsymbol{G} \gets 
            \Call{TMult}{{E_{i\times j} a {\tau}^{r'-r}},
                {\boldsymbol{A_{n}^{-1}}}}$ \\
            \Comment{we are setting the leading coefficient of $\boldsymbol{G}[i, j]$\\
          \qquad\qquad \qquad\qquad\qquad\quad  to be the same as leading coefficient of $\boldsymbol{V}[i, j]$}
            \State $\boldsymbol{G} \gets \Call{TMult}{{\boldsymbol{G}},
            {\Phi}} - \Call{TMult}{{\Psi}, {\boldsymbol{G}}}$
            \State $\boldsymbol{V} \gets \boldsymbol{V}-\boldsymbol{G} $
            \State $r' \gets r' - 1$
        \EndWhile
    \EndFor
\EndFor
\State \Return $\boldsymbol{V}$
\EndFunction

\end{algorithmic}
\end{algorithm}

\begin{algorithm}[H]
\caption{}\label{alg:helper3}
\begin{algorithmic}[1]

\Function{CoefficientForm}{$\boldsymbol{V}, {d}$}

\Comment{lists all coefficients up to degree defined by matrix $d$ for each\\
\qquad\qquad  $\tau$ polynomial in $\boldsymbol{V}$ in order defined by columns }

\Input
\Desc{$\boldsymbol{V}$}{matrix of $\tau$ polynomials}
\Desc{$d$}{matrix of integers with the same dimensions as $\boldsymbol{V}$}
\EndInput
\Output
\Desc{$\boldsymbol{A}$}{column matrix of coefficients}
\EndOutput

\State $n \gets \Call{Rows}{\boldsymbol{V}}$
\State $m \gets \Call{Columns}{\boldsymbol{V}}$
\State $s \gets 1$

\For {$i = 1..n$}
    \For {$j = 1..m$}
        \For {$k = 0..d[i,j]-1$}
        \algstore{myalg3}
\end{algorithmic}
\end{algorithm}

\begin{algorithm}
\begin{algorithmic}
\algrestore{myalg3}
      
            \State $\boldsymbol{A}[s, 1] \gets
            \Call{Coefficient}{{\boldsymbol{V}[i,j]}, {k}} $
            \State $s \gets s+1$
        \EndFor    
    \EndFor
\EndFor
\State \Return $\boldsymbol{A}$
\EndFunction

\end{algorithmic}
\end{algorithm}

\begin{algorithm}[H]
\caption{Computing extension of t-modules with inverse matrix}\label{alg:algorithm1}
\begin{algorithmic}[1]
\Function{Extension1}{$\Phi$, $\Psi$}

\Input
\Desc{$\Phi$, $\Psi$}{\qquad t-modules}
\EndInput
\Output
\Desc{$\Pi$}{\qquad extension of $\Phi$ by $\Psi$}
\EndOutput

\State $n \gets \Call{Dim}{\Psi}$
\State $m \gets \Call{Dim}{\Phi}$
\State $r \gets \Call{Deg}{\Phi}$
\State $d \gets [r]_{n \times m}$
\Comment{$n \times m$ matrix with all elements equal to $r.$ \\ \qquad \qquad\qquad\qquad\qquad\qquad
 We do not have to construct this matrix \\ \qquad \qquad\qquad\qquad\qquad\qquad here, but it generalizes nicely for the next \\ \qquad \qquad\qquad\qquad\qquad\qquad algorithm}
\State $\boldsymbol{A_n}$ : $m \times m$ matrix \\
 \Comment{ computing inverse matrix used  later for \\ \qquad \qquad\qquad\qquad\qquad\qquad\quad\,\,  degree reduction}
\For {$i = 1..m$}
   \For {$j = 1..m$}
        \If{$\Call{Deg}{\Phi \lbrack i, j \rbrack} = r$}
        \State $\boldsymbol{A_n}[i, j] \gets \Call{Coefficient}{{\Phi \lbrack i, j \rbrack}, {r}}$
        \Else
        \State $\boldsymbol{A_n}[i, j] \gets 0$
        \EndIf
    \EndFor
\EndFor

\State $\boldsymbol{A_{n}^{-1}} \gets \Call{Inverse}{\boldsymbol{A_n}}$
\State $s \gets 1$

\For {$i = 1..n$}
    \Comment{computing columns of $\Pi$}
    \For {$j = 1..m$}
        \For {$k = 1..r$}
            \State $\boldsymbol{V_{s}}
                \gets \Call{TMult}{
                    {\Psi}, E_{i\times j} c {\tau}^{k-1}}$
            \State $\boldsymbol{V_{s}} \gets
                \Call{Reduce1}{
                {\boldsymbol{V_{s}}}, {\Phi}, {\Psi}, {\boldsymbol{A_{n}^{-1}}}}
            $
            \State $\boldsymbol{V_{s}} \gets
             \Call{CoefficientForm}{\boldsymbol{V_{s}}, {d}}$
            \State $\Call{Substitute}{\boldsymbol{V_{s}}, c^{(\_)} \to {\tau}^{(\_)}}$
            \State $s \gets s + 1$
        \EndFor
    \EndFor
\EndFor

\State $\Pi \gets \begin{bmatrix}
    \boldsymbol{V_{1}} & ... & \boldsymbol{V_{s-1}} 
    \end{bmatrix}$
\State \Return $\Pi$

\EndFunction

\end{algorithmic}
\end{algorithm}


 \section{Algorithm in case of a composition series}\label{sec5}
Since the category of \tm modules is not abelian  then  for a \tm module $\Upsilon$ and its \tm submodule $\Psi$
their quotient $\Upsilon/\Psi$ is usually not a \tm module (although it is an ${\mathbb F}_q[t]-$module since the category 
of \tm modules is a subcategory of the category of ${\mathbb F}_q[t]-$modules).
This motivates the following definitions:
\begin{definition}\label{submodule}
We say that a \tm module $\Psi$ is a $\tau-$ submodule of a \tm module $\Upsilon$ iff there exists a short 
exact sequence of \tm modules:
$$0\lra \Psi\lra \Upsilon\lra \Upsilon/\Psi\lra 0.$$
\end{definition}
It is clear that the dimension of a $\tau-$submodule $\Psi$ is always smaller than the dimension $\Upsilon$ 
    
\begin{definition}\label{submodule}
We say that a \tm module $\Psi$ is  $\tau-$simple if  it does not have any non-trivial $\tau -$ submodules i.e. it is not a middle term of any nontrivial short exact sequence of  \tm modules with nonzero terms.
\end{definition}	
Directly  from the 
Definition \ref{submodule} it follows that Drinfeld modules are $\tau-$simple.
On the other hand any \tm module $\Upsilon$ for which under some isomorphism the matrix ${\Upsilon}_t$ can be reduced to the lower triangular matrix is not $\tau$-simple. This easily follows from the description of an extension by means of a biderivation.
\begin{example}\label{ex1}
The \tm module given by the matrix
$$\Upsilon_t=\left[\begin{array}{cc}
	\theta & \tau\\
	\tau & \theta
	
\end{array}\right]\cong \left[\begin{array}{cc}
\theta+\tau & 0\\
\tau & \theta-\tau
\end{array}\right]$$
is not $\tau$-simple since the isomorphism 
$$f=\left[\begin{array}{cc}
	1&1\\
	0& 1
\end{array}\right].$$ 
brings $\Upsilon_t$ to the lower diagonal form.
\end{example} 

\begin{definition}
  By a $\tau-$composition  series of a \tm module $\Phi$ we mean a sequence of \tm modules:
  \begin{equation}\label{comp}
  0=\Phi_0\subset  \Phi_1 \subset  \Phi_2 \subset \cdots \subset  \Phi_{n-1}\subset  \Phi_n=\Phi
  \end{equation}
  such that ${\Phi}_i\subset{\Phi}_{i+1}$ is an inclusion of a $\tau -$ submodule and the quotient 
  \tm module ${\Phi}_{i+1}/{\Phi}_{i}$ is $\tau -$ simple.
  \end{definition}
 Let us  give a simple proof of the following:
\begin{thm}\label{existence}
Any \tm module is either $\tau -$simple or has non-trivial, i.e. $n\geq 2$,   $\tau -$ composition series   (\ref{comp}).
\end{thm}
\begin{proof}
Notice that existence  of $\tau$-composition series:
	 
\begin{align*}
		0=\Upsilon_0&\subset \Upsilon_1 \subset \cdots \subset \Upsilon_{n-1} \subset \Upsilon_n=\Upsilon,\\
	\end{align*}
where the quotients $\sigma_i=\Upsilon_{i}/\Upsilon_{i-1}$ are  $\tau-$simple \tm modules is equivalent with the fact that  $\Upsilon$  is isomorphic with  a \tm module given by the block lower-triangular matrix where the blocks on the diagonal 
correspond to subsequent   $\tau-$simple modules
 $\sigma_{n}$, $\sigma_{n-1}$,$\dots$, $\sigma_{1}$. 
 So it is enough to show that every \tm module is given by the block lower-triangular matrix where on the diagonal one 
 has $\tau -$simple modules.
 If \tm module $\Upsilon$ is $\tau -$simple then there is nothing to prove.
 We use induction on dimension  $d$ of the non-simple \tm module. 
 Assume that $\Upsilon$ is not $\tau -$simple. In particular $d>1.$
 If $d=2$ then we the following exact sequence:
 $$0\rightarrow \sigma_{1} \rightarrow \Upsilon \rightarrow \sigma_{2} \rightarrow 0$$
 where $\sigma_{1}$ and $\sigma_{2}$ are Drinfeld modules.
  Therefore $\Upsilon$ is given by the following matrix: 
$$\Upsilon_t=\left[\begin{array}{c|c}
	\sigma_{2,t} & 0\\ \hline
	\delta_t & \sigma_{1,t}
\end{array}\right],$$
where $\delta_t$ is a biderivation. This shows the assertion for $d=2.$
Assume that the theorem holds true for all 
 \tm modules of dimension less than  $d$. Let $\Upsilon$ be of dimension $d$. Since $\Upsilon$ is not $\tau-$simple   
 there exists a short exact sequence of the form:
$0\lra \Phi\lra \Upsilon \lra \Psi\lra 0$. 
Then
\begin{equation}\label{Ut}
\Upsilon_t=\left[\begin{array}{c|c}
	\Psi_t & 0\\ \hline
	\delta_t & \Phi_t
\end{array}\right],
\end{equation}
for certain biderivation $\delta$. Since dimensions of $\Phi$ and $\Psi$ are less than $d$ by inductive hypothesis we can consider $\Phi_t$ and  $\Psi_t$ as block lower-triangular matrices with  $\tau-$simple  \tm modules on diagonals. 
Thus the same holds true for $\Upsilon_t$.   
\end{proof}

\begin{example}\label{prod}
The product $\Pi_{i=1}^n\phi_i$ of  Drinfeld modules  $\phi_i$ has the following $\tau -$composition series:
\begin{equation}\label{prod1}
	0\subset \phi_1\subset \phi_1\times \phi_2\subset \cdots \subset \Pi_{i=1}^{n-1}\phi_i
	\subset \Pi_{i=1}^{n}\phi_i,
	\end{equation}
	where quotients of consecutive $\tau -$submodules are Drinfeld modules (which are $\tau -$simple \tm modules). 
\end{example}
\begin{thm}
	\label{thm:ciagi_kompozycyjne}
	Let $\Upsilon$ and $\widehat{\Upsilon}$ be \tm modules having the following composition series: 
		\begin{align*}
		0=\Upsilon_0&\subset \Upsilon_1 \subset \cdots \subset \Upsilon_{n-1} \subset \Upsilon_n=\Upsilon,\\
		0=\widehat{\Upsilon}_0&\subset \widehat{\Upsilon}_1 \subset \cdots \subset \widehat{\Upsilon}_{m-1} \subset \widehat{\Upsilon}_m=\widehat{\Upsilon},\\
	\end{align*}
	such that for $i=1,2,\cdots, n$ and $j=1,2,\cdots, m$ the following conditions are fulfilled:
	\begin{itemize}
		\item[$(i)$] the quotients $\sigma_i=\Upsilon_i/\Upsilon_{i-1}$  and 
		$\widehat{\sigma}_j=\widehat{\Upsilon}_j/\widehat{\Upsilon}_{j-1}$ are $\tau$-simple \tm modules;
		\item[$(ii)$]  $\Hom_{\tau}(\sigma_i,\widehat{\sigma}_j)=0, \,\,i=1,\dots ,n; \,\,j=1,\dots ,m$; 
		\item[$(iii)$] for every pair $(i,j),\,\,i\in\{1,2,\dots,n\}$ and $j\in\{1,2,\dots,m\}$ the spaces $\Ext^1_{\tau}(\sigma_i, \widehat{\sigma}_j)$ have  \tm module structures coming from 
		\tm reduction.
	\end{itemize}
	Then the space $\Ext^1_{\tau}(\Upsilon, \widehat{\Upsilon})$ has a \tm module structure coming from \tm reduction.  
\end{thm}
\begin{proof}
For $i=1,2,\dots, n$ and $j=1,2,\dots, m$	we introduce the following notation for exact sequences of \tm modules:
	\begin{align*}
		\eta_{i}:&\quad 0\lra \Upsilon_{i-1} \lra \Upsilon_i\lra \sigma_i\lra 0\\
		\widehat{\eta}_{j}:&\quad 0\lra \widehat{\Upsilon}_{j-1} \lra \widehat{\Upsilon}_j\lra \widehat{\sigma}_j\lra 0.\\
	\end{align*}
Applying the functor
 $\Hom_{\tau}( -,\widehat{\sigma}_j )$ to the sequence $\eta_2$  by Theorem \ref{thm:long_sequence} we obtain  the following six term exact sequence:
\begin{align*}
	0&\lra \nadwzorem{\Hom_{\tau}( \sigma_2,\widehat{\sigma}_j )}{=0}\lra \Hom_{\tau}( \Upsilon_2,\widehat{\sigma}_j )\lra \nadwzorem{\Hom_{\tau}( \Upsilon_1,\widehat{\sigma}_j )}{=0}\lra \\
	&\lra \Ext^1_{\tau}(  \sigma_2,\widehat{\sigma}_j )\lra \Ext^1_{\tau}( \Upsilon_2 ,\widehat{\sigma}_j )\lra \Ext^1_{\tau}( \Upsilon_1,\widehat{\sigma}_j )\lra 0.\\
\end{align*}
Then $\Hom_{\tau}(\Upsilon_2,\widehat{\sigma_j})=0$ and applying Lemma \ref{lem:algorytm_reduckji_dla_krotkiego_ciagu} to the lower exact sequence we see that the space   $\Ext^1_{\tau}( \Upsilon_2 ,\widehat{\sigma}_j )$ has a \tm module structure coming from \tm reduction. 
Proceeding analogously for  $\eta_3, \eta_4, \dots, \eta_n$ we obtain that 
$\Hom_{\tau}(\Upsilon,\widehat{\sigma_j})=0$ and the space
$\Ext^1_{\tau}( \Upsilon ,\widehat{\sigma}_j )$  has a \tm module structure  coming from \tm reduction. 
Applying the functor $\Hom_{\tau}(\Upsilon,-)$ to the sequence $\widehat{\eta}_2$  again by Theorem \ref{thm:long_sequence} we obtain the following six term exact sequence:
\begin{align*}
	0&\lra \nadwzorem{\Hom_{\tau} (\Upsilon, \widehat{\Upsilon}_1)}{=0}\lra   \Hom_{\tau} (\Upsilon, \widehat{\Upsilon}_2)\lra \nadwzorem{\Hom_{\tau} (\Upsilon, \widehat{\sigma}_2)}{=0}\lra \\
	&\lra\Ext^1_{\tau}(\Upsilon, \widehat{\Upsilon}_1)\lra  \Ext^1_{\tau}(\Upsilon, \widehat{\Upsilon}_2)\lra  \Ext^1_{\tau}(\Upsilon, \widehat{\sigma}_2)\lra 0
\end{align*}
Thus $\Hom_{\tau} (\Upsilon, \widehat{\Upsilon}_2)=0$  and applying again Lemma \ref{lem:algorytm_reduckji_dla_krotkiego_ciagu} to the lower row we obtain that  
$  \Ext^1_{\tau}(\Upsilon, \widehat{\Upsilon}_2)$ has a \tm module structure coming from \tm reduction. 
Proceeding in this way for sequences  $\widehat{\eta_j},\, j=3,4,\dots,m$ we finally obtain that 
 $\Ext^1_{\tau}(\Upsilon, \widehat{\Upsilon})$ has a  \tm module structure coming from \tm reduction. 
\end{proof}

We have the following corollary of Theorem \ref{thm:ciagi_kompozycyjne}:
\begin{cor}\label{cor:ciagi_kompozycyjne}
	Let $\Upsilon$ and $\widehat{\Upsilon}$ be \tm modules with the following $\tau-$ composition series: 
	\begin{align*}
		0=\Upsilon_0&\subset \Upsilon_1 \subset \cdots \subset \Upsilon_{n-1} \subset \Upsilon_n=\Upsilon,\\
		0=\widehat{\Upsilon}_0&\subset \widehat{\Upsilon}_1 \subset\cdots \subset \widehat{\Upsilon}_{m-1} \subset\widehat{\Upsilon}_m=\widehat{\Upsilon},\\
	\end{align*}
where quotients $\sigma_i=\Upsilon_i/\Upsilon_{i-1}$  and 
$\widehat{\sigma}_j=\widehat{\Upsilon}_j/\widehat{\Upsilon}_{j-1}$ are Drinfeld modules such that  
$\rk\sigma_i>\rk\widehat{\sigma}_j$ for every pair $(i,j),\,\,i\in\{1,2,\dots,n\}$ and $j\in\{1,2,\dots,m\}$. Then 
the space 
$\Ext^1_{\tau}(\Upsilon, \widehat{\Upsilon})$ has a structure of a \tm module coming from \tm reduction. 
\end{cor}
\begin{remark}
Notice that Corollary \ref{cor:ciagi_kompozycyjne} is a generalization of \cite[Theorem 6.1.]{kk04} which asserts that $\Ext^1_{\tau}$ for products of  Drinfeld modules, with the appropriate assumption on  degrees, has a  \tm module structure.
\end{remark}

\begin{algorithm}[H]
\caption{}\label{alg:helper3}
\begin{algorithmic}[1]

\Function{Reduce2}{${\boldsymbol{V}}, {\Phi}, {\Psi}$}

\Comment{Reduction of t-module for lower triangular t-modules}

\Input
\Desc{$\boldsymbol{V}$}{module to be reduced}
\Desc{$\Phi$}{t-module denoted as $\Upsilon$ in Corollary \ref{cor:ciagi_kompozycyjne}}
\Desc{$\Psi$}{t-module denoted as $\widehat\Upsilon$ in Corollary \ref{cor:ciagi_kompozycyjne}}
\EndInput

\Output
\Desc{$\boldsymbol{V}$}{reduced module}
\EndOutput

\State $n \gets \Call{Rows}{\boldsymbol{V}}$
\State $m \gets \Call{Columns}{\boldsymbol{V}}$

  \Comment{we reduce degree of $\tau$ polynomial at position $i$, $j$\\ \qquad\qquad\qquad\qquad\quad in 
        $\boldsymbol{V}$ to degree less than $\Phi [j, j]$, columns are  redu-\\ \qquad\qquad\qquad\qquad\quad ced  in
        reverse order}

\For {$i = 1..n$}
    \For {$j = m..1$}
      
        \State $r' \gets \Call{Deg}{\boldsymbol{V}[i,j]}$
        \State $r \gets \Call{Deg}{\Phi [j, j]}$
 
    \algstore{myalg5}
\end{algorithmic}
\end{algorithm}

\begin{algorithm}
\begin{algorithmic}
\algrestore{myalg5}       
     
      \While {$ r' \geq r $}
            \State $a \gets \frac{\Call{Coefficient}{\boldsymbol{V}[i,j], r'}}
            {(\Call{Coefficient}{\Phi [j,j], r})^{(r'-r)}} $
            \State $\boldsymbol{G} \gets 
            {E_{i\times j} a {\tau}^{r'-r}} $
            \State $\boldsymbol{G} \gets \Call{TMult}{{\boldsymbol{G}},
            {\Phi}} - \Call{TMult}{{\Psi}, {\boldsymbol{G}}}$
            \State $\boldsymbol{V} \gets \boldsymbol{V}-\boldsymbol{G} $
            \State $r' \gets r' - 1$
        \EndWhile                
    \EndFor
\EndFor
\State \Return $\boldsymbol{V}$
\EndFunction
\\
\end{algorithmic}
\end{algorithm}

\begin{algorithm}[H]
\caption{Computing extension of triangular t-modules}\label{alg:algorithm2}
\begin{algorithmic}[1]

\Function{Extension2}{$\Phi$, $\Psi$}

\Input
\Desc{$\Phi$, $\Psi$}{t-modules}
\EndInput

\Output
\Desc{$\Pi$}{extension of $\Phi$ by $\Psi$}
\EndOutput

\State $n \gets \Call{Dim}{\Phi}$
\State $m \gets \Call{Dim}{\Psi}$
\State $d$ : $n \times m$ matrix

\For {$i = 1..n$}
\Comment{computing degrees of basis elements}
\For {$j = 1..m$}
\State $d[i,j] \gets \Call{Deg}{\Phi[j,j]}$
\EndFor
\EndFor

\State $s \gets 1$
\For {$i = 1..n$}
\Comment{computing columns of $\Pi$}
    \For {$j = 1..m$}
        \State $r \gets \Call{Deg}{\Phi[j,j]}$
        \For {$k = 1..r$}
            \State $\boldsymbol{V_{s}}
                \gets \Call{TMult}{{\Psi}, {E_{i\times j} c {\tau}^{k-1}}}$
            \State $\boldsymbol{V_{s}} \gets
                \Call{Reduce2}
                {{\boldsymbol{V_{s}}}, {\Phi}, {\Psi}}
            $
            \State $\boldsymbol{V_{s}} \gets
                \Call{CoefficientForm}{{\boldsymbol{V_{s}}}, {d}}$
            \State $\Call{Substitute}{\boldsymbol{V_{s}}, c^{(\_)} \to {\tau}^{(\_)}}$
            \State $s \gets s + 1$                 
        \EndFor
    \EndFor
\EndFor
\State $\Pi \gets \begin{bmatrix}
    \boldsymbol{V_{1}} & ... & \boldsymbol{V_{s-1}} 
    \end{bmatrix}$
\State \Return $\Pi$

\EndFunction
\end{algorithmic}
\end{algorithm}


\section{Exact formulas for $\Ext^1$ for two Drinfeld modules}\label{sec6}
  In most cases analyzed in this paper, it is unreasonable to expect getting exact formulas for $\Ext^1$ since they are too complicated. 
In this section we present exact formulas, for the easiest case, describing 
 \tm module structure on $\Ext^1_{\tau}(\phi, \psi)$ for two Drinfeld modules.
So assume that $\phi$ and $\psi$ are Drinfeld modules, such that $\phi_t=\theta+\sum\limits_{i=1}^na_i\tau^i$, $\psi_t=\theta+\sum\limits_{j=1}^mb_j\tau^j$ and $r:=n-m>0$.  
We transfer the $\F_q[t]-$module structure from $\Ext^1_{\tau}(\phi, \psi)=\Der(\phi, \psi)/\Derin(\phi, \psi)$ to the space $K\{\tau\}_{<\rk\phi}$, via the isomorphism from Lemma 3.1 of \cite{kk04}. Since each element of $K\{\tau\}_{<\rk\phi}$ is of the form $\sum\limits_{i=0}^{n-1}c_i\tau^i$, it is sufficient to determine the values of  multiplication by $t$ on the generators $c_i\tau^i$ for $i=0,1,\dots, n-1$, where $c_i\in K$. Our reduction algorithm implies that the \tm module structure on 
	$\Ext^1_{\tau}(\phi, \psi)$ is given by the following matrix:
		\begin{equation}\label{matrix_case_1}
		\Pi_t=\Bigg[\big[t*\tau^{0}\big],\big[t*\tau^{1}\big],\dots,\big[t*\tau^{r-1}\big],\big[t*\tau^{r}\big],\big[t*\tau^{r+1}\big],\dots, \big[t*\tau^{n-1}\big]\Bigg].
	\end{equation}
 defining the map $\Pi:\F_q[t]\lra {\mathrm{M}}_n(K)\{\tau\}.$  	 In order to describe $\big[t*\tau^{i}\big]$ we introduce the following notation: $b_k\tau^k|_{c_i}:= \up{c_i}kb_k$ for $k=1,\dots, m$ and $i=1,\dots, \rk\phi-1.$ Let $e_i=\tau^i$ for $i=0,1,\cdots, n-1$ be the basis of
	 $K\{\tau\}_{<\rk\phi}.$ Then  the column vector $\big[t*\tau^{i}\big]$ is defined by the following property:
	$t*(c_i\tau^{i})=\big[t*\tau^{i}\big]|_{c_i}.$
	 
	 \textbf{In the case $m>r$}  we obtain the following formula
	 \begin{align*}
		\big[t*\tau^{i}\big]&=  \big[\podwzorem{ 0,\dots, 0}{i-1}, \theta, b_1\tau, b_2\tau^2,\dots, b_m\tau^m,0,\dots,0 \big]^t &&\textnormal{for}&& i=0,1,2,\dots r-1.
	\end{align*} 
	 	
	Fix $l\in\{0,1,\dots, m-1\}$ and define recursively the following polynomials:
	\begin{align}\label{polynomials_d_case_1}
		d_{r+l,n+k}(\tau)=\dfrac{b_{m-l+k}\tau^{m-l+k} -\sum\limits_{i=k+1}^{l}\up{a_{n+k-i}}{i}d_{r+l,n+i}(\tau)  }{\up{a_n}k} 
	\end{align} 
	for $k=l, l-1, \dots, 1,0.$
	Then we can express
	$\big[t*\tau^{r+l}\big]$ in an explicit way:
	\begin{align*}
		\big[t*&\tau^{r+l}\big]=\\
		=& \Bigg[0, \Big[ \big(\theta - \up{\theta} {j}\big)d_{r+l,n+j}(\tau) +	 
		\sum\limits_{k=0}^{j-1}\big( b_{j-k}\up{d_{r+l,n+k}}{j-k}(\tau)  - \up{a_{j-k}}k d_{r+l,n+k}(\tau) \big)   \Big]_{j=1}^{l} \\
		&,\Big[\sum\limits_{k=0}^{l}\Big( b_{j-k}\up{d_{r+l,n+k}}{j-k}(\tau)  - \up{a_{j-k}}k d_{r+l,n+k}(\tau)    \Big) \Big]_{j=l+1}^{m-1},\\
		&,\Big[ \sum\limits_{k=j-m}^{l} b_{j-k}\up{d_{r+l,n+k}}{j-k}(\tau)   - \sum\limits_{k=0}^{l} \up{a_{j-k}}k d_{r+l,n+k}(\tau)    \Big]_{j=m}^{m+l}, \\	
		&, \Big[ - \sum\limits_{k=0}^{l} d_{r+l,n+k} \up{a_{j-k}}k   \Big]_{j=m+l+1}^{r+l-1},
		\theta - \sum\limits_{k=0}^{l} \up{a_{r+l-k}}k d_{r+l,n+k}(\tau), \\
		&, \Big[ b_{j-r-l}\tau^{j-r-l}- \sum\limits_{k=0}^{l} \up{a_{j-k}}k d_{r+l,n+k}(\tau) \Big]_{j=r+l+1}^{n-1}\Bigg]^t.
	\end{align*}
	
	\begin{remark}
		Notice that in the above formula the notation $[f(j)]_{j=u}^v$ means $[f(u)],[f(u+1)],\dots [f(v)].$
	\end{remark}
	
	\textbf{In the case $m=r$} we obtain the following coordinates
	\begin{align*}
		\big[t*\tau^{i}\big]&=  \big[\podwzorem{ 0,\dots, 0}{i-1}, \theta, b_1\tau, b_2\tau^2,\dots, b_m\tau^m,0,\dots,0 \big]^t &&\textnormal{for}&& i=1,2,\dots r-1.
	\end{align*} 
	\begin{align*}
		\big[t*\tau^{r+l}\big]=&\Bigg[0,\Big[\big(\theta -\up\theta j\big)d_{r+l,n+1}(\tau)
		+\sum\limits_{k=0}^{j-1}\Big(b_{j-k}\up{d_{r+l,n+k}}{j-k}(\tau)-\up {a_{j-k}}kd_{r+l,n+k}(\tau) \Big)\Big]_{j=1}^l,\\
		&,  \Big[\sum\limits_{k=0}^{l}\Big(b_{j-k}\up{d_{r+l,n+k}}{j-k}(\tau)-\up {a_{j-k}}kd_{r+l,n+k}(\tau) \Big)\Big]_{j=l+1}^m, \\
		&,\Big[\sum\limits_{k=j-m}^{l}b_{j-k} \up{d_{r+l,n+k}}{(j-k)}(\tau) - \sum\limits_{k=0}^{l}\up {a_{j-k}}kd_{r+l,n+k}(\tau)\Big]_{j=m+1}^{r+l-1},\\
		&, \Big[ \theta+ b_m\up{d_{r+l,n+l}}m(\tau)  - 
		\sum\limits_{k=0}^{l}\up {a_{r+l-k}}k d_{r+l,n+k}(\tau)\Big],\\
		&, \Big[b_{j-r-l}\tau^{j-r-l} - 
		\sum\limits_{k=0}^{l}\up {a_{j-k}}kd_{r+l,n+k}(\tau) \Big]_{j=r+l+1}^{n-1}
		\Bigg]^t,
	\end{align*} 
	where the polynomials $d_{r+l,n+k}(\tau)$ are the same as before, see (\ref{polynomials_d_case_1}).\\
		\textbf{In the case $m>r$} we obtain the following coordinates
	\begin{align*}
		\big[t*\tau^{i}\big]&=  \big[\podwzorem{ 0,\dots, 0}{i-1}, \theta, b_1\tau, b_2\tau^2,\dots, b_m\tau^m,0,\dots,0 \big]^t &&\textnormal{for}&& i=1,2,\dots r-1.
	\end{align*} 
	To describe the remaining coordinates we must consider two sub-cases: $r+l\leq m$ or $r+l>m$. \\
	In the sub-case $r+l\leq m$ we have:
	\begin{align*}
		\big[t*\tau^{r+l}\big]&=\Bigg[0,\big(\theta -\up\theta j\big)\Big[d_{r+l,n+j}(\tau)+ \sum\limits_{k=0}^{j-1}\Big(\up{b_{j-k}d_{r+l,n+k}}{j-k}(\tau)-\up {a_{j-k}}k d_{r+l,n+k}(\tau) \Big)\Big]_{j=1}^{l}\\
		&,\Big[\sum\limits_{k=0}^{l}\Big(b_{j-k}\up{d_{r+l,n+k}}{j-k}(\tau)-\up {a_{j-k}}k d_{r+l,n+k}(\tau) \Big)\Big]_{j=l+1}^{r+l-1}, \\
		&+\Big[\theta+ \sum\limits_{k=0}^{l}\Big(\up{b_{r+l-k}d_{r+l,n+k}}{r+l-k}(\tau)-\up {a_{r+l-k}}k d_{r+l,n+k}(\tau) \Big)\Big], \\
		&,\Big[b_{j-r-l}\tau^{j-r-l} + \sum\limits_{k=0}^{l}\Big(b_{j-k}\up{d_{r+l,n+k}}{j-k}(\tau)-\up{a_{j-k}}k d_{r+l,n+k}(\tau) \Big)\Big]_{j=r+l+1}^{m}, \\
		&,\Big[b_{j-r-l}\tau^{j-r-l} + \sum\limits_{k=j-m}^{l}b_{j-k}\up{d_{r+l,n+k}}{j-k}(\tau) - \sum\limits_{k=0}^{l}\up {a_{j-k}}k d_{r+l,n+k}(\tau)
		\Big]_{j=m+1}^{m+l}\\
		&,\Big[b_{j-r-l}\tau^{j-r-l} - 
		\sum\limits_{k=0}^{l}\up {a_{j-k}}k d_{r+l,n+k}(\tau)\Big]_{j=m+l+1}^{n-1}\Bigg]^t,
	\end{align*} 
	where the polynomials $d_{r+l,n+k}(\tau)$ are given recursively by the formula: 
	\begin{align}\label{polynomial_d_case_3}
		d_{r+l,n+k}(\tau)=&\dfrac{b_{m-l+k}\tau^{m-l+k}-\sum\limits_{i=k+1}^{l}\up{a_{n+k-i}}{i}d_{r+l,n+i}(\tau)  +\sum\limits_{i=k+r}^{l}b_{n+k-i}\up{d_{r+l,n+i}}{n+k-i}(\tau)}{\up{a_n}{k}} 
	\end{align} 
	for $k=l, l-1, \dots, 1,0$.\\
	In the sub-case $r+l> m$ we have:
	\begin{align*}
		\big[t*\tau^{r+l}\big]&=\Bigg[0,\Big[\big(\theta -\up\theta j\big)d_{r+l,n+j}(\tau)+ \sum\limits_{k=0}^{j-1}\Big(b_{j-k}\up{d_{r+l,n+k}}{j-k}(\tau)-\up {a_{j-k}}kd_{r+l,n+k}(\tau) \Big)\Big]_{j=1}^{l}\\
		&,\Big[\sum\limits_{k=0}^{l}\Big(b_{j-k}\up{d_{r+l,n+k}}{j-k}(\tau)-\up {a_{j-k}}k d_{r+l,n+k}(\tau) \Big)\Big]_{j=l+1}^{m}, \\
		&,\Big[ \sum\limits_{k=j-m}^{l}\up{b_{j-k} d_{r+l,n+k}}{j-k}(\tau)  - \sum\limits_{k=0}^{l}\up {a_{j-k}}k d_{r+l,n+k}(\tau)
		\Big]_{j=m+1}^{r+l-1},\\
		&,\Big[\theta+ \sum\limits_{k=r+l-m}^{l}b_{r+l-k}\up{d_{r+l,n+k}}{r+l-k}(\tau) - \sum\limits_{k=0}^{l}\up {a_{r+l-k}}k d_{r+l,n+k}(\tau)\Big], \\
		&,\Big[b_{j-r-l}\tau^{j-r-l} 
		+\sum\limits_{k=j-m}^{l}b_{j-k}\up{d_{r+l,n+k}}{j-k}(\tau) - \sum\limits_{k=0}^{l}\up {a_{j-k}}k d_{r+l,n+k}(\tau) \Big]_{j=r+l+1}^{m+l},\\
		&,\Big[b_{j-r-l}\tau^{j-r-l}  
		- \sum\limits_{k=0}^{l}\up {a_{j-k}}kd_{r+l,n+k}(\tau)\Big]_{j=m+l+1}^{n-1}\Bigg]^t,
	\end{align*}  
	where the polynomials $d_{r+l,n+k}(\tau)$ are given by (\ref{polynomial_d_case_3}).
From the calculated forms of $\Pi_t$, we  observe that 
	\begin{itemize}
		\item[$(i)$] The polynomials $d_{*,*}(\tau)$ used in the description of $\Pi_t$ do not have constant terms. 
		\item[$(ii)$] If we write $\Pi_t$ in the form 
		$\Pi_t=(I\theta +N)\tau^0 +\sum\limits_{i=1}^sA_i\tau^i,$
		then the matrix $N$ vanishes. 
	\end{itemize}

\subsection{Some consequences of exact formulas}
\begin{prop}
	Let $(K,v)$ be a valuation field. 
	Assume that $\phi$ and $\psi$ are Drinfeld modules, such that  $r:=\rk\phi-\rk\psi>0$ and 
	$\phi_t=\theta+\sum\limits_{i=1}^na_i\tau^i$, $\psi_t=\theta+\sum\limits_{j=1}^mb_j\tau^j$.
	If the following conditions
	\begin{itemize}
		\item[$(i)$] $\psi_t$ has integer coefficients,
		\item[$(ii)$] $\phi_t-a_n\tau^n$ has integer coefficients,
		\item[$(iii)$] $v(a_n)<\min\Bigg\{v\big(b_1\big), q^{i}v\big(a_{n-i}\big), \dfrac {v\big(b_{m+1-i}\big)}{q^{m-i}}   \mid i=1,2,\dots, m-1\Bigg\}$,
	\end{itemize}
	hold true,	then \tm module $\Ext^1_{\tau}(\phi, \psi)$ has integer coefficients.  
\end{prop} 
\begin{proof}
	First we will show, that  the condition $(iii)$ implies that the polynomials $d_{r+l, n+k}(\tau)$ have integer coefficients. The proof will be carried out  for the case $m>r$. The other cases are similar.  Fix $l\in\{0,1,\dots, m-1\}$.    We use the downward induction for $k=l,l-1,\dots, 1,0$. \\
	If $k=l$, then $d_{r+l, n+l}(\tau)=\dfrac{b_m}{\up{a_n}l}\tau^m.$ Because $v(a_n)<\dfrac{v(b_m)}{q^m}$, then 
	$$v(b_m)>q^mv(a_n)>q^{l}v(a_n)=v(\up{a_n}l),$$ 
	so the polynomial $d_{r+l, n+l}(\tau)$ has integer coefficients. Assume that the polynomials 
	$$d_{r+l, n+l}(\tau), d_{r+l, n+l-1}(\tau), \cdots,  d_{r+l, n+k+1}(\tau)$$
	have integer coefficients. We have 
	\begin{align*}
		d_{r+l,n+k}(\tau)&=\dfrac{b_{m-l+k}\tau^{m-l+k}-\sum\limits_{i=k+1}^{l}\up{a_{n+k-i}}{i} d_{r+l,n+i}(\tau)  +\sum\limits_{i=k+r}^{l}b_{n+k-i}\up{d_{r+l,n+i}}{n+k-i}(\tau)}{\up{a_n}{k}} \\
		&= \dfrac{b_{m-l+k}}{\up{a_n}{k}}\tau^{m-l+k}-\sum\limits_{i=k+1}^{l}
		\dfrac { \up {a_{n+k-i}}{i}}{\up{a_n}{k}} d_{r+l,n+i}(\tau) +\sum\limits_{i=k+r}^{l}\dfrac{b_{n+k-i}}{\up{a_n}{k}}\up{d_{r+l,n+i}}{n+k-i}(\tau).
	\end{align*} 
	From the inductive hypotheses the polynomials $d_{r+l,n+i}(\tau)$ and $\up{d_{r+l,n+i}}{n+k-i}(\tau)$ have 
	integer coefficients. Thus, it is enough to show that the valuations of fractions 
	$$\dfrac{b_{m-l+k}}{\up{a_n}{k}},\quad \dfrac { \up {a_{n+k-i}}{i}}{\up{a_n}{k}}, \quad \dfrac{b_{n+k-i}}{\up{a_n}{k}}   $$
	are positive. This follows from condition $(iii)$. Hence the polynomials $d_{r+l,n+k}(\tau)$ have  integer coefficients. \\
	The claim follows now from the form of the matrix $\Pi_t$ of the \tm module $\Ext^1_{\tau}(\phi, \psi)$ by using  conditions $(i)$ and $(ii)$.
\end{proof}
\begin{prop}
	Let $\phi$ and $\psi$ be  Drinfeld modules, such that $r=\rk\phi-\rk\psi>0$. Then we have:
	$$\rk\Ext^1_{\tau}(\phi, \psi)=\begin{cases} 2m \quad {\mathrm{for}}\,\,m\leq r\\
		3m \quad {\mathrm{for}} \,\,m>r  \end{cases}.$$
\end{prop}
\begin{proof}
	Claim follows  from the form of $\Pi_t$ in the corresponding cases.
\end{proof}

\newpage	
\section*{Appendix}

In this Appendix we give  an exemplary  implementation of the pseudocodes from Sections \ref{sec4} and  \ref{sec5}. This is done by means of  Mathematica 13.2 \cite{w23}. We inserted a lot of comments in the program so that the reader can easily follow it. We also compute two examples: first using Algorithm 4 and the second using Algorithm 6.
 
 \medskip

\noindent\( \pmb{\text{(*}\,\, \text{functions}\,\, \text{used} \,\,\text{for}\,\, \text{getting} \,\,\text{proper} \,\,\text{formatting} \,\,\text{for}\,\, q{}^{\wedge}i \,\,\text{*)}}\\
 {\text{par}[\text{x$\_$}]\text{:=}\text{RowBox}[\{\text{{``}({''}},x,\text{{``}){''}}\}]\text{//}\text{DisplayForm}}\\
 {\text{formatQ}[\text{p$\_$}]\text{:=}p \text{/.} (q{}^{\wedge}\text{x$\_$}\text{-$>$}\text{par}[x])\text{/.}(q\text{-$>$}\text{par}[1]) }\\
 \pmb{\text{(* degree of $\tau $ polynomial *)}}\\
 {\deg [\text{p$\_$}] \text{:=} \text{Max}[\text{Exponent}[p, \tau ], 0]}\\
 \pmb{\text{(* degree of $\tau $ polynomial with matrix coefficients *)}}\\
 {\text{degM}[\text{m$\_$}]\text{:=}\text{Max}[\text{Exponent}[\#, \tau ]\&\text{/@}m];}\\
 {}
 \pmb{\text{(*} \,\,\text{functions}  \,\,\text{used} \,\, \text{for} \,\, \text{computing}  \,\,\text{product}  \,\,\text{of}  \,\,\text{matrices}  \,\,x, y  \,\,\text{of}  \,\,\tau   \,\,\text{polynomials} \,\,
\text{*)}}\\
 {\text{twist}[\text{c$\_$}, \text{n$\_$}] \text{:=} \text{Table}[c {}^{\wedge} (q {}^{\wedge} (i - 1)), \{i, n\}]}\\
 {\text{toPoly}[\text{p$\_$}] \text{:=} \text{FromDigits}[\text{Reverse}[p], \tau ]}\\
 {\text{coefficientList}[\text{p$\_$}] \text{:=} \text{If}[\text{Exponent}[p, \tau ] \text{$>$=} 0, \text{CoefficientList}[p, \tau ], \{0\}]}\\
 {\text{pMult}[\text{x$\_$}, \text{y$\_$}] \text{:=} \text{Total}[\text{toPoly} \text{/@} (\text{MapIndexed}[\text{PadLeft}[\text{twist}[\text{$\#$1},
\deg [x}\\
 {\text{  }] + 1] * \text{coefficientList}[x], \text{$\#$2} + \deg [x]]\&, \text{CoefficientList}[y, \tau ]] \text{/.} 0{}^{\wedge}\text{u$\_$}
\text{-$>$}}
 {\text{   }0)];}\\
 {\text{tMult}[\text{x$\_$}, \text{y$\_$}]\text{:=}\text{Expand}[\text{Inner}[\text{pMult}, x, y, \text{Plus}]];}\\
 {}
 {\text{\bf(*} \,\,{E_{i, j}} \,\,\text{\bf matrix} \,\,\text{\bf with} \,\,\text{\bf dimensions}\,\, n \times m \,\,\text{\bf *)}}\\
 {\text{elementMatrix}[\text{i$\_$},\text{j$\_$},\text{n$\_$},\text{m$\_$}]\text{:=}\text{Module}[\{x\},x =\text{Table}[0, \{p,n\},\{r,m\}]; x[[i]][[j]]=1;x];}\\
 {}
\pmb {\text{(*  \,\,array  \,\,of  \,\,$\tau $  \,\, polynomial  \,\, coefficients  \,\, of  \,\, each  \,\, element  \,\,of  \,\,v  \,\,*)}}\\
{\text{coefficientForm}[\text{v$\_$},\text{degrees$\_$}]
\text{:=}\text{Flatten}[\text{MapIndexed}[\text{Table}[\text{Coefficient}[\text{$\#$1},\tau
,k],}\\
{\text \{k,0,\text{Extract}[\text{degrees},\text{$\#$2}]-1\}]\& ,v,\{2\}],2];}\)\\
\noindent \( \pmb{\text{(*} \phi , \psi  - t-\text{modules} \text{*)}}\\
 \pmb{\text{(*} \,\,\text{reduction}\,\, \text{of} \,\,\text{degrees}\,\, \text{in}\,\, v \,\,\text{to}\,\, \text{degrees} \,\,\text{at} \,\,\text{most}\,\, r-1, \,\,\text{using} \,\,\text{inverse}\,\,
\text{matrix}\,\, s \,\,\text{*)}}\\
 {\text{reduce1}[\text{v$\_$},\text{r$\_$},\phi \_,\psi \_,\text{s$\_$}]\text{:=}\text{Module}[\{i,j,n,m,\text{rp},a,g,\text{vp}\},}\\
 {\text{vp}=v;} \\
 {n = \text{Dimensions}[v][[1]];}\\
 {m = \text{Dimensions}[v][[2]];}\\
 {\text{For}[i=1,i\text{$<$=}n,i\text{++},}\\
 {\text{For}[j=1,j\text{$<$=}m,j\text{++},}\\
 {\text{rp}=\deg [\text{vp}[[i]][[j]]];}\\
 {\text{While}[\text{rp}\text{$>$=}r,}\\
 {a=\text{Coefficient}[\text{vp}[[i]][[j]],\tau ,\text{rp}];}\\
 {g=\text{tMult}[\text{elementMatrix}[i,j,n,m]*a*\tau {}^{\wedge}(\text{rp}-r),s];}\\
 {g=\text{tMult}[g,\phi ]-\text{tMult}[\psi ,g];}\\
 {\text{vp} =\text{PowerExpand}[\text{vp}-g];}\\
 {\text{rp}=\text{rp}-1;}
 {]}
 {]}
 {];}
\,\, {\text{vp}}
 {];}\\
\pmb {\text{(*}\,\, \text{computing} \,\,\text{extension} \,\,\text{Ext}^1(\phi ,\psi ) \,\,\text{using} \,\,\text{inverse} \,\,\text{matrix} \,\,\text{*)}}\\
 {\text{extInverse}[\phi \_,\psi \_]\text{:=}\text{Module}[\{n, m, r, s, v, \text{degrees},c\},}\\
 {m = \text{Dimensions}[\phi ][[1]];}\\
 {n = \text{Dimensions}[\psi ][[2]];}\\
 {r= \text{degM}[\phi ];}\\
 {\text{degrees}=\text{Table}[r,\{i,n\},\{j,m\}];}\\
\pmb {\text{(*} \,\,\text{finding} \,\,\text{A$\_$n}{}^{\wedge}-1\,\, \text{*)}}\\
 {s= \text{Inverse}[\text{Map}[\text{Coefficient}[\#,\tau , r]\&, \phi , \{2\}]];}\\
 \pmb{\text{(*} \,\,\text{creating} \,\,\text{basis} \,\,\text{for} \,\,\text{extension}, \,\,\text{where}\,\, v \,\,\text{is} \,\,a \,\,\text{list} \,\,\text{of} \,\,\text{basis} \,\,\text{elements}\,\,
\text{*)}}\\
 {v=\text{Flatten}[\text{Table}[\text{elementMatrix}[i,j,n,m]*c*\tau {}^{\wedge}(k-1),\{i,n\},\{j,m\},\{k,r\}],2];}\\
 {v=\text{tMult}[\psi ,\#]\& \text{/@}v;}
 {}\\
 {v=\text{reduce1}[\#,r,\phi ,\psi ,s]\&\text{/@}v;}\\
 {v=\text{coefficientForm}[\#,\text{degrees}]\&\text{/@}v;}\\
\pmb {\text{(* expand powers and substitute c for $\tau $ with appropriate power *)}}\\
 {v=\text{PowerExpand}[v\text{//.}(\text{x$\_$}+\text{y$\_$}){}^{\wedge}(q{}^{\wedge}\text{z$\_$})\text{-$>$}x{}^{\wedge}(q{}^{\wedge}z)+y{}^{\wedge}(q{}^{\wedge}z)\text{//.}(\text{x$\_$}+\text{y$\_$}){}^{\wedge}q\text{-$>$}x{}^{\wedge}q+y{}^{\wedge}q];}\\
 {v = v\text{//.}(\text{x$\_$}+\text{y$\_$})*\text{z$\_$}\text{-$>$}x*z+y*z;}\\
 {v = v\text{/.}c{}^{\wedge}(q{}^{\wedge}\text{x$\_$})\text{-$>$}\tau {}^{\wedge}x\text{/.}c{}^{\wedge}(q)\text{-$>$}\tau \text{/.}c\text{-$>$}1;}\\
 {v \text{//}\text{Transpose}}
 {];}
 {}\\
 \pmb{\text{(*} \text{example} - \text{Ext}^1(\phi ,\psi ) \text{*)}}\\
 {\psi  =\{\{\theta +\tau {}^{\wedge}2,0\}, \{b,\theta +\tau \}\};}\\
 {\phi  = \{\{\theta , \tau {}^{\wedge}3\}, \{\tau {}^{\wedge}3+a, \theta \}\};}\\
 {\phi  = \{\{\theta , \tau {}^{\wedge}3\}, \{\tau {}^{\wedge}3+a, \theta \}\};}\\
 {\psi \text{//}\text{MatrixForm}}\\
 {\phi \text{//}\text{MatrixForm}}\\
 {\text{formatQ} \text{/@}\text{extInverse}[\phi ,\psi ]\text{//} \text{MatrixForm}}\)

\noindent\(\left(
\begin{array}{cc}
 \theta +\tau ^2 & 0 \\
 b & \theta +\tau  \\
\end{array}
\right)\)

\noindent\(\left(
\begin{array}{cc}
 \theta  & \tau ^3 \\
 a+\tau ^3 & \theta  \\
\end{array}
\right)\)

\noindent\(\left(
\begin{array}{cccccccccccc}
 \theta  & -a \tau ^2 & 0 & 0 & 0 & 0 & 0 & 0 & 0 & 0 & 0 & 0 \\
 0 & \theta  & -a^{(1)} \tau ^2 & 0 & 0 & \theta  \tau ^2-\theta ^{(1)} \tau ^2 & 0 & 0 & 0 & 0 & 0 & 0 \\
 \tau ^2 & 0 & \theta +\tau ^6 & 0 & \tau ^4 & 0 & 0 & 0 & 0 & 0 & 0 & 0 \\
 0 & 0 & 0 & \theta  & 0 & 0 & 0 & 0 & 0 & 0 & 0 & 0 \\
 0 & 0 & \theta  \tau ^2-\theta ^{(1)} \tau ^2 & 0 & \theta  & 0 & 0 & 0 & 0 & 0 & 0 & 0 \\
 0 & \tau ^4 & 0 & \tau ^2 & 0 & \theta  & 0 & 0 & 0 & 0 & 0 & 0 \\
 b & 0 & b \tau ^4 & 0 & b \tau ^2 & 0 & \theta  & 0 & -a \tau  & 0 & 0 & 0 \\
 0 & b & 0 & 0 & 0 & b \tau ^2 & \tau  & \theta  & 0 & 0 & 0 & \tau ^2 \\
 0 & 0 & b & 0 & 0 & 0 & 0 & \tau  & \theta  & 0 & 0 & 0 \\
 0 & b \tau ^2 & 0 & b & 0 & 0 & 0 & 0 & 0 & \theta  & 0 & 0 \\
 0 & 0 & b \tau ^2 & 0 & b & 0 & 0 & 0 & \tau ^2 & \tau  & \theta  & 0 \\
 0 & 0 & 0 & 0 & 0 & b & 0 & 0 & 0 & 0 & \tau  & \theta  \\
\end{array}
\right)\)

\noindent\( \pmb{\text{(*} \phi , \psi  - t-\text{modules} \text{*)}}\\
\pmb {\text{(* reduction of degrees in v *)}}\\
 {\text{reduce2}[\text{v$\_$},\phi \_,\psi \_]\text{:=}\text{Module}[\{i,j,n,m,\text{rp},r,a,g,\text{vp}\},}\\
 {\text{vp}=v;}\qquad
 {n = \text{Dimensions}[v][[1]];}\qquad
 {m = \text{Dimensions}[v][[2]];}\\
 {\text{For}[i=1,i\text{$<$=}n,i\text{++},}\\
 {\text{For}[j=m,j\text{$>$=}1,j\text{--},}\\
 {\text{rp}=\deg [\text{vp}[[i]][[j]]];}\qquad
 {r = \deg [\phi [[j]][[j]]];}\\
 {\text{While}[\text{rp}\text{$>$=}r,}\\
 {a=\text{Coefficient}[\text{vp}[[i]][[j]],\tau ,\text{rp}]/(\text{Coefficient}[\phi [[j]][[j]],\tau ,r]{}^{\wedge}(q{}^{\wedge}(\text{rp}-r)));}\\
 {g=\text{elementMatrix}[i,j,n,m]*a*\tau {}^{\wedge}(\text{rp}-r);}\\
 {g=\text{tMult}[g,\phi ]-\text{tMult}[\psi ,g];}\\
 {\text{vp} =\text{PowerExpand}[\text{vp}-g];}\\
 {\text{rp}=\text{rp}-1;}
 {]}
 {]}
 {];}\qquad
 {\text{vp}}
 {];}\\
 \pmb{\text{(*}\,\, \text{computing} \,\,\text{extension} \,\,\text{Ext}^1(\phi ,\psi ) \,\,\text{for} \,\,\text{lower} \,\,\text{triangular}\,\, t-\text{modules} \,\,\text{*)}}\\
 {\text{extTriangular}[\phi \_,\psi \_]\text{:=}\text{Module}[\{n,m,v,\text{degrees},x,y,z,c\},}\\
 {m = \text{Dimensions}[\phi ][[1]];}\\
 {n = \text{Dimensions}[\psi ][[2]];}
 {}\\
 {\text{degrees}=\text{Table}[\deg [\phi [[j]][[j]]],\{i,n\},\{j,m\}];}\\
 \pmb{\text{(*} \,\,\text{creating}\,\, \text{basis} \,\,\text{for} \,\,\text{extension},\,\, \text{where} \,\,v \,\,\text{is} \,\,a \,\,\text{list} \,\,\text{of} \,\,\text{basis} \,\,\text{elements}\,\,
\text{*)}}\\
 {v=\text{Flatten}[\text{Table}[}\\
 {\text{Table}[\text{elementMatrix}[i,j,n,m]*c*\tau {}^{\wedge}(k-1),}\\
 {\{k,\deg [\phi [[j]][[j]]]\}],\{i,n\},\{j,m\}],2];}\\
 {v=\text{tMult}[\psi ,\#]\&\text{/@}v;}
 {}\\
 {v=\text{reduce2}[\#,\phi ,\psi ]\&\text{/@}v;}\\
 {v=\text{coefficientForm}[\#,\text{degrees}]\&\text{/@}v;}\\
 \pmb{\text{(* expand powers and substitute c for $\tau $ with appropriate power *)}}\\
 {v=\text{PowerExpand}[v\text{//.}(\text{x$\_$}+\text{y$\_$}){}^{\wedge}(q{}^{\wedge}\text{z$\_$})\text{-$>$}x{}^{\wedge}(q{}^{\wedge}z)+y{}^{\wedge}(q{}^{\wedge}z)\text{//.}(\text{x$\_$}+\text{y$\_$}){}^{\wedge}q\text{-$>$}x{}^{\wedge}q+y{}^{\wedge}q];}\\
 {v = v\text{//.}(\text{x$\_$}+\text{y$\_$})*\text{z$\_$}\text{-$>$}x*z+y*z;}\\
 {v = v\text{/.}c{}^{\wedge}(q{}^{\wedge}\text{x$\_$})\text{-$>$}\tau {}^{\wedge}x\text{/.}c{}^{\wedge}(q)\text{-$>$}\tau \text{/.}c\text{-$>$}1;}\\
 {v \text{//}\text{Transpose}}
 {];}
 {}\\ 
\pmb {\text{(*} \,\,\text{example} - \text{Ext}^1(\phi ,\psi ),\, \,\,\text{char k}\neq 3\,\,\text{*)}}\\
 {\psi  =\{\{\theta +\tau ,0\}, \{1,\theta +\tau \}\};}\\
 {\phi  = \{\{\theta +\tau {}^{\wedge}3, 0\}, \{1+a*\tau +\tau {}^{\wedge}2, \theta +3*\tau {}^{\wedge}2\}\};}\\
 {\psi \text{//}\text{MatrixForm}}\\
 {\phi \text{//}\text{MatrixForm}}\\
 {\text{formatQ} \text{/@}\text{extTriangular}[\phi ,\psi ] \text{//} \text{MatrixForm}}\)

\medskip

\noindent\(\left(
\begin{array}{cc}
 \theta +\tau  & 0 \\
 1 & \theta +\tau  \\
\end{array}
\right)\)

\noindent\(\left(
\begin{array}{cc}
 \theta +\tau ^3 & 0 \\
 1+a \tau +\tau ^2 & \theta +3 \tau ^2 \\
\end{array}
\right)\)

\noindent\(\left(
\begin{array}{cccccccccc}
 \theta  & 0 & 0 & 0 & -\frac{\tau }{3} & 0 & 0 & 0 & 0 & 0 \\
 \tau  & \theta  & \tau ^2 & 0 & -\frac{a \tau }{3} & 0 & 0 & 0 & 0 & 0 \\
 0 & \tau  & \theta  & 0 & -\frac{\tau }{3} & 0 & 0 & 0 & 0 & 0 \\
 0 & 0 & 0 & \theta  & 0 & 0 & 0 & 0 & 0 & 0 \\
 0 & 0 & 0 & \tau  & \theta +3^{-(1)} \tau ^2 & 0 & 0 & 0 & 0 & 0 \\
 1 & 0 & \tau  & 0 & 0 & \theta  & 0 & 0 & 0 & -\frac{\tau }{3} \\
 0 & 1 & 0 & 0 & 0 & \tau  & \theta  & \tau ^2 & 0 & -\frac{a \tau }{3} \\
 0 & 0 & 1 & 0 & 0 & 0 & \tau  & \theta  & 0 & -\frac{\tau }{3} \\
 0 & 0 & 0 & 1 & \frac{\tau }{3} & 0 & 0 & 0 & \theta  & 0 \\
 0 & 0 & 0 & 0 & 1 & 0 & 0 & 0 & \tau  & \theta +3^{-(1)} \tau ^2 \\
\end{array}
\right)\)
{\text{(* \bf recall that $(n):=q^{\large n}$ *)}}\\

\vspace{15mm}

\centerline{\bf Corrigendum}

In the paper \emph{Algorithms for determination of $t$-module structures on some extension groups}, there is a local error in Algorithm~2, more precisely in the function \texttt{Reduce1}, in the case of strictly pure \tm modules.

The error appears both in the printed pseudocode and in the corresponding \emph{Mathematica} implementation given in the Appendix.  In certain cases, the function \texttt{Reduce1} does not completely reduce the matrix $\boldsymbol{V}$.

In the published version, the reduction of $\boldsymbol{V}$ starts at the upper-left entry and proceeds row by row, from left to right. However, when reducing the entry $(i,j)$, it may happen that the degree of an earlier entry $(k,l)$ with $k \leq i$ and $l \leq j$ increases. In such a situation, the original function does not return to that entry and, therefore, leaves it unreduced.

We stress that this error is local. Both the underlying reduction procedure and the proof of correctness remain valid. In particular, all proofs and conclusions presented in the paper remain correct. The correction concerns only the explicit form of the procedure \texttt{Reduce1} for strictly pure $t$-modules, as well as the corresponding implementation.

The corrected version of the pseudocode is as follows.
\setcounter{algorithm}{1}
\begin{algorithm}[H]
\caption{Corrected version of \textsc{Reduce1}}
\label{alg:helper2}
\begin{algorithmic}[1]

\Function{Reduce1}{${\boldsymbol{V}}, {\Phi}, {\Psi}, {\boldsymbol{A_{n}^{-1}}}$}
\Comment{Reduction of t-module using \\ \qquad\qquad\qquad\qquad\quad\qquad\qquad\qquad\qquad\quad\,\, inverse matrix}

\Input
\Desc{$\boldsymbol{V}$}{module to be reduced}
\Desc{$\Phi$} { strictly pure t-module }
\Desc{$\Psi$}{ a t-module of degree less than degree $\Phi$}
\Desc{$\boldsymbol{A_{n}^{-1}}$}{inverse of the leading matrix of $\Phi$}
\EndInput
\Output
\Desc{$\boldsymbol{V}$}{reduced module}
\EndOutput

\State $n \gets \Call{Rows}{\boldsymbol{V}}$
\State $m \gets \Call{Cols}{\boldsymbol{V}}$
\State $r \gets \Call{Deg}{\Phi}$

\For {$i = 1..n$}
    \For {$j = 1..m$}
      
        \State $r' \gets \Call{Deg}{\boldsymbol{V}[i,j]}$

        \If {$ r' \geq r $}
            \State $a \gets \Call{Coefficient}{\boldsymbol{V}[i,j], r'} $
            \State $\boldsymbol{G} \gets 
            \Call{TMult}{{E_{i\times j} a {\tau}^{r'-r}},
                {\boldsymbol{A_{n}^{-1}}}}$ \\
            \Comment{we are setting the leading coefficient of $\boldsymbol{G}[i, j]$\\
          \qquad\qquad \qquad\qquad\qquad\quad  to be the same as leading coefficient of $\boldsymbol{V}[i, j]$}
          \algstore{myalg3}
\end{algorithmic}
\end{algorithm}

\begin{algorithm}
\begin{algorithmic}
\algrestore{myalg3}
            \State $\boldsymbol{G} \gets \Call{TMult}{{\boldsymbol{G}},
            {\Phi}} - \Call{TMult}{{\Psi}, {\boldsymbol{G}}}$
            \State $\boldsymbol{V} \gets \boldsymbol{V}-\boldsymbol{G} $
            \State $i \gets 1$
            \State $j \gets 0$
        \EndIf
    \EndFor
\EndFor
\State \Return $\boldsymbol{V}$
\EndFunction

\end{algorithmic}
\end{algorithm}

Thus, whenever a reduction is performed, the procedure restarts the traversal from the entry $(1,1)$. This ensures that any earlier entry whose $\tau$-degree becomes greater than or equal to $r=\deg_{\tau}\Phi$
during a later reduction step is not left unreduced.

Accordingly, the implementation in the Appendix  should also be replaced by the following corrected version:

\begin{doublespace}
\noindent {\text{reduce1}[\text{v$\_$},\text{r$\_$},$\phi \_$,$\psi \_$,\text{s$\_$}]\text{:=}\text{Module}[\{i,j,n,m,\text{rp},a,g,\text{vp}\},}\\
 {\text{vp}=v;} \\
 {n = \text{Dimensions}[v][[1]];}\\
 {m = \text{Dimensions}[v][[2]];}\\
 {\text{For}[i=1,i\text{$<$=}n,i\text{++},}\\
 {\text{For}[j=1,j\text{$<$=}m,j\text{++},}\\
 {\text{rp}=$\deg$ [\text{vp}[[i]][[j]]];}\\
 {\text{If}[\text{rp}\text{$>$=}r,}\\
 {a=\text{Coefficient}[\text{vp}[[i]][[j]],$\tau$ ,\text{rp}];}\\
 {g=\text{tMult}[\text{elementMatrix}[i,j,n,m]*a*$\tau$ ${}^{\wedge}$(\text{rp}-r),s];}\\
 {g=\text{tMult}[g,$\phi$ ]-\text{tMult}[$\psi$ ,g];}\\
 {\text{vp} =\text{PowerExpand}[\text{vp}-g];}\\
 {\text{i}=1;}
 {\text{j}=0;}
 {]}
 {]}
 {];}
\,\, {\text{vp}}
 {];}
\end{doublespace}

As a consequence, Example~4.1 and the example in the Appendix  should also be corrected accordingly.
\setcounter{section}{4}
\setcounter{rem}{0}
\begin{example}
 Let $$\Phi_t=\left[\begin{array}{cc}
		\theta & {\tau}^3\\
		1+{\tau}^3 & \theta
	\end{array}\right]\quad \textnormal{and}\quad 
	\Psi_t=\left[\begin{array}{cc}\theta+\tau^2 & 0\\
		1 & \theta+\tau
	\end{array}\right].$$    
Then $N_{\Phi}=N_{\Psi}=\left[\begin{array}{cc}
		0 & 0\\
		1 & 0
	\end{array}\right],\quad \textnormal{and} \quad A_3^{-1}=\left[\begin{array}{cc}
		0 & 1\\
		1 & 0
	\end{array}\right].$ One readily verifies that
$E_{i\times j}A_3^{-1}N_{\Phi}=N_{\Psi} E_{i\times j}A_3^{-1}$ only for $(i,j)=(2,2). $ Thus $s=1.$ Accordingly, the reduction algorithm yields
$$\Ext^1_{\tau}(\Phi,\Psi)=\left[\begin{array}{cccccccccccc}
		\theta & -{\tau}^2 & 0 & 0 & 0 & -\tau^4 & 0 & 0 & 0 & 0 & 0 & 0 \\
0 &\theta & -{\tau}^2 & 0 & 0 & (\theta -{\theta}^{(1)}){\tau}^2 & 0 & 0 & 0 & 0 & 0 & 0 \\	
{\tau}^2 & 0 & \theta+{\tau}^6 & 0 & {\tau}^4 & 0 & 0 & 0 & 0 & 0 &0 & 0 \\
0 & 0 & 0 & \theta & 0 & 0 & 0 & 0 & 0 & 0 &0 & 0 \\
0 & 0 & (\theta -{\theta}^{(1)}){\tau}^2 & 0 & \theta & 0 & 0 & 0 & 0 & 0 &0 & 0 \\
0 & {\tau}^4 & 0 & {\tau}^2 & 0 & \theta+\tau^6 & 0 & 0 & 0 & 0 &0 & 0 \\
1 & 0 & {\tau}^4 & 0 & {\tau}^2 & 0 & \theta & 0 & -{\tau} & 0 & 0 &0  \\
0 & 1 & 0 & 0 & 0 & {\tau}^2 & \tau & \theta & 0 & 0 & 0 & {\tau}^2 \\
0 & 0 & 1 & 0 & 0 & 0 & 0 & \tau  & \theta & 0 & 0 & 0  \\
0 & {\tau}^2 & 0 & 1 & 0 & \tau^4 & 0 & 0 & 0  & \theta & 0 & 0  \\
0 & 0 & {\tau}^2 & 0 & 1 & 0 & 0 & 0 & {\tau}^2 & \tau  & \theta & 0   \\
0 & 0 & 0 & 0 & 0 & 1 & 0 & 0 & 0 & 0   & \tau & \theta   \\
    \end{array}\right]$$
\end{example}

\begin{example}[Corrected example in the Appendix]\quad \\
\begin{doublespace}
\noindent\(
 \pmb{\text{(*} \text{example} - \text{Ext}^1(\phi ,\psi ) \text{*)}}\\
 {\psi  =\{\{\theta +\tau {}^{\wedge}2,0\}, \{b,\theta +\tau \}\};}\\
 {\phi  = \{\{\theta , \tau {}^{\wedge}3\}, \{\tau {}^{\wedge}3+a, \theta \}\};}\\
 {\phi  = \{\{\theta , \tau {}^{\wedge}3\}, \{\tau {}^{\wedge}3+a, \theta \}\};}\\
 {\psi \text{//}\text{MatrixForm}}\\
 {\phi \text{//}\text{MatrixForm}}\\
 {\text{formatQ} \text{/@}\text{extInverse}[\phi ,\psi ]\text{//} \text{MatrixForm}}\)

\noindent\(\left(
\begin{array}{cc}
 \theta +\tau ^2 & 0 \\
 b & \theta +\tau  \\
\end{array}
\right)\)
\end{doublespace}

\begin{doublespace}
\noindent\(\left(
\begin{array}{cc}
 \theta  & \tau ^3 \\
 a+\tau ^3 & \theta  \\
\end{array}
\right)\)
\end{doublespace}

\begin{doublespace}
\noindent\(\left(
\begin{array}{cccccccccccc}
 \theta  & -a \tau ^2 & 0 & 0 & 0 & -a\tau^4 & 0 & 0 & 0 & 0 & 0 & 0 \\
 0 & \theta  & -a^{(1)} \tau ^2 & 0 & 0 & (\theta  -\theta ^{(1)}) \tau ^2 & 0 & 0 & 0 & 0 & 0 & 0 \\
 \tau ^2 & 0 & \theta +\tau ^6 & 0 & \tau ^4 & 0 & 0 & 0 & 0 & 0 & 0 & 0 \\
 0 & 0 & 0 & \theta  & 0 & 0 & 0 & 0 & 0 & 0 & 0 & 0 \\
 0 & 0 & (\theta-\theta ^{(1)}) \tau ^2 & 0 & \theta  & 0 & 0 & 0 & 0 & 0 & 0 & 0 \\
 0 & \tau ^4 & 0 & \tau ^2 & 0 & \theta+\tau^6  & 0 & 0 & 0 & 0 & 0 & 0 \\
 b & 0 & b \tau ^4 & 0 & b \tau ^2 & 0 & \theta  & 0 & -a \tau  & 0 & 0 & 0 \\
 0 & b & 0 & 0 & 0 & b \tau ^2 & \tau  & \theta  & 0 & 0 & 0 & \tau ^2 \\
 0 & 0 & b & 0 & 0 & 0 & 0 & \tau  & \theta  & 0 & 0 & 0 \\
 0 & b \tau ^2 & 0 & b & 0 & b\tau^4 & 0 & 0 & 0 & \theta  & 0 & 0 \\
 0 & 0 & b \tau ^2 & 0 & b & 0 & 0 & 0 & \tau ^2 & \tau  & \theta  & 0 \\
 0 & 0 & 0 & 0 & 0 & b & 0 & 0 & 0 & 0 & \tau  & \theta  \\
\end{array}
\right)\)
\end{doublespace}
\end{example}

To summarize, the correction concerns:
\begin{itemize}
    \item Algorithm~2, function \texttt{Reduce1}, in the case of strictly pure $t$-modules,
    \item the corresponding \emph{Mathematica} implementation in the Appendix,
    \item Example~4.1,
    \item the example in the Appendix.
\end{itemize}

All proofs and all conclusions of the paper remain correct.

\end{document}